\documentclass[journal]{IEEEtran}
\usepackage[figuresright]{rotating}
\usepackage{amssymb}
\usepackage{amsthm}
\usepackage{multicol}
\usepackage{subfigure}
\usepackage{graphicx}
\usepackage{epstopdf}
\usepackage{fullpage}
\usepackage{latexsym,amsmath}
\usepackage{stmaryrd}
\usepackage{algorithm,algorithmic}
\usepackage{subfigure}
\usepackage{amsfonts}
\usepackage{xcolor,multirow}
\usepackage{cite}
\usepackage{bm}
\usepackage{url}
\usepackage{diagbox}
\usepackage{array}
\usepackage{booktabs}
\usepackage{footmisc} 
\usepackage{pifont}
\usepackage{caption}
\usepackage{mathrsfs}
\usepackage{graphicx}
\usepackage{makecell}
\usepackage{subcaption}  
\usepackage[normalem]{ulem}
\usepackage{adjustbox} 

\renewcommand{\arraystretch}{1.25}

\newtheorem{definition}{Definition}
\newtheorem{remark}{Remark}
\newtheorem{property}{Property}

\newtheorem{proposition}{Proposition}

\ifCLASSINFOpdf
\else
\fi
\begin{document}
	
	\title{Tensor Orthogonal Subspace Split: Theory and Applications}

	\author{Jifei~Miao, Juan~Han, Michael~K.~Ng, ~\IEEEmembership{Senior Member,~IEEE}, and Kit Ian Kou
		\thanks{Jifei Miao is with the School of Mathematics and Statistics, Yunnan University, Kunming, Yunnan, 650091, China (e-mail: jifmiao@163.com
			).}
		\thanks{Juan Han is with the School of Mathematics and Physics, Anhui Jianzhu University, Hefei, Anhui, 230601, China (e-mail:
			juanhan0604@163.com).}
		\thanks{Michael K. Ng is with the Department of Mathematics, Hong Kong
			Baptist University, Kowloon Tong, Hong Kong, China (e-mail: michael-ng@hkbu.edu.hk).
		}
		\thanks{Kit Ian Kou is with the Department of Mathematics, Faculty of
			Science and Technology, University of Macau, Macau 999078, China
			(e-mail: kikou@umac.mo).}
	}

	
	\maketitle
	
	\begin{abstract}
		Tensor representations have emerged as a fundamental paradigm for modeling multidimensional data by preserving intrinsic correlations across multiple modes. This paper proposes a novel theoretical framework, termed Tensor Orthogonal Subspace Split (TOSS), which explicitly splits a tensor, along a prescribed mode, into two orthogonal components: \emph{a dominant component} lying in a prescribed subspace and \emph{a residual component} lying in the corresponding orthogonal complement. We first present the general formulation of TOSS and systematically investigate its fundamental properties. As an important and practically meaningful special case, we further introduce the rank-one TOSS, which imposes a separable rank-one structure along the splitting mode and admits a clear geometric interpretation. This formulation naturally captures dominant consistent patterns while effectively isolating orthogonal residual component. The proposed framework establishes a unified theoretical foundation for tensor-domain orthogonal split and opens new avenues for structured tensor modeling across diverse applications. Building upon the developed TOSS theory, hyperspectral image restoration and color video background modeling are considered as two representative tasks, for which corresponding optimization models are formulated. Efficient algorithms are developed to solve the resulting problems. Extensive experimental results validate the effectiveness and superiority of the proposed approaches.
	\end{abstract}
	
	
	\IEEEpeerreviewmaketitle
	
	\section{Introduction}
	Multidimensional data naturally arise in numerous scientific and engineering applications, including hyperspectral imaging, video sequences, medical imaging, and multi-sensor systems. Such data are commonly represented as high-order tensors to preserve intrinsic spatial, spectral, and temporal structures \cite{auddy2025tensors, han2023multi}. Exploiting these multi-way correlations has become a central topic in signal processing and machine learning, leading to extensive research on tensor decomposition, low-rank modeling, and structured regularization \cite{TOKCAN2026110191,heng2023robust,kolda2009tensor, cichocki2015tensor}.

	Existing tensor-based methods have extensively exploited low-rank priors through various tensor decomposition and relaxation techniques, including Tucker decomposition \cite{de2000multilinear}, CANDECOMP/PARAFAC (CP) factorization \cite{harshman1970foundations}, tensor train (TT) decomposition \cite{oseledets2011tensor}, tensor ring (TR) decomposition \cite{zhao2016tensor}, tensor wheel decomposition \cite{wu2022tensor}, fully-connected tensor network decomposition \cite{ zheng2021fully}, as well as recent adaptive tensor network decomposition frameworks that aim to automatically adjust network structures or ranks according to the underlying data characteristics \cite{nie2021adaptive,nie2023adaptive,liu2024adaptively}. In addition, tensor nuclear norm regularization under different algebraic frameworks, particularly those based on tensor singular value decomposition (t-SVD), has also been widely studied \cite{kilmer2013third,lu2019tensor,lu2016tensor}. These approaches have achieved remarkable success in a wide range of reconstruction, completion, and denoising tasks by promoting compact representations of high-dimensional data. Despite their success, most existing methods mainly emphasize compact representation and low-rank approximation, rather than the explicit characterization of structurally different information components in tensor data.
	\begin{figure*}[htbp]
		\centering
		\includegraphics[width=16.3cm, height=9cm]{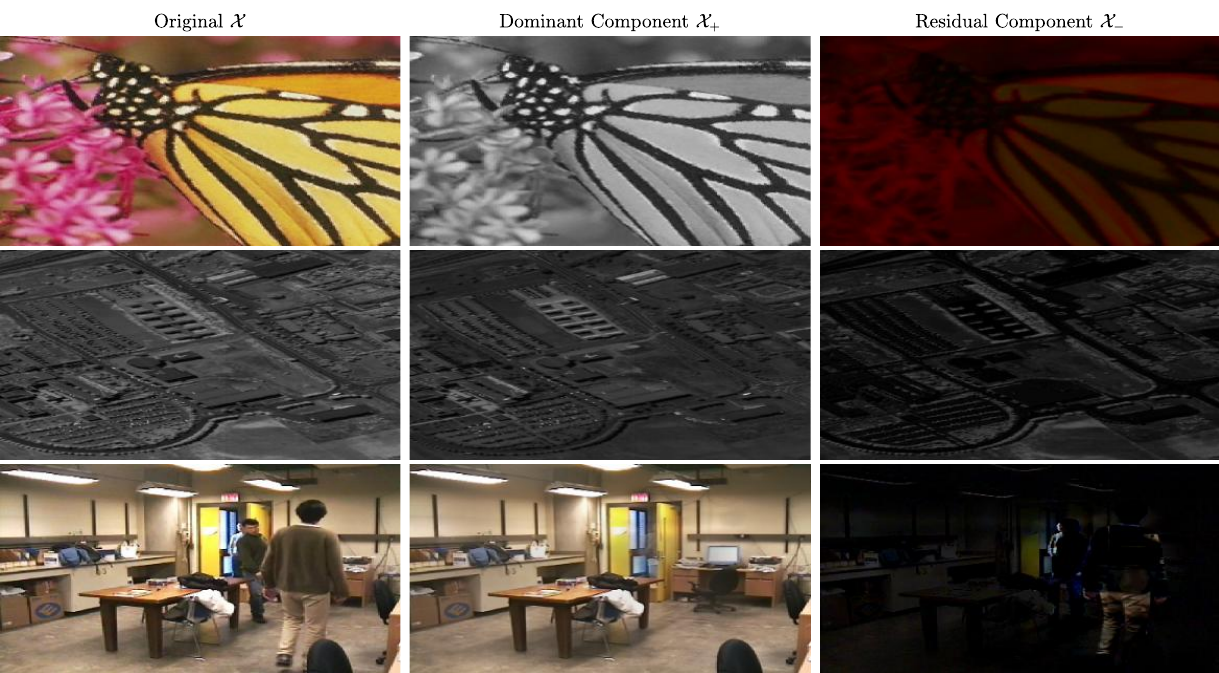}
		\caption{Illustrative examples of the rank-one TOSS mechanism. The first row shows a color image, the second row shows one band of a hyperspectral image, and the third row shows one frame of a color video. The first column corresponds to the original data $\mathcal{X}$, the second column represents the dominant component $\mathcal{X}_{+}$, and the third column represents the residual component $\mathcal{X}_{-}$ that is orthogonal to the dominant component.}
		\label{TOSS_exam}
	\end{figure*}
	
	In real high-order tensor data, an important observation is that different information components often possess markedly different structural roles. Generally, the dominant part of the data is governed by a common latent direction or a low-dimensional subspace shared by tensor fibers along a prescribed mode, thus forming stable and consistent patterns with clear physical or semantic meaning, as illustrated by $\mathcal{X}_{+}$ in Fig. \ref{TOSS_exam}. In contrast, the remaining part typically reflects deviations from such consistent structure and is often characterized by fine details, localized variations, degradations, perturbations, or noise, as illustrated by $\mathcal{X}_{-}$ in Fig. \ref{TOSS_exam} \cite{bioucas2012hyperspectral,yang2025subspace}. However, this type of mode-wise structural consistency and its orthogonal deviation cannot be explicitly or adequately captured by existing tensor decomposition methods or conventional low-rank models. Therefore, beyond compact representation or low-rank approximation, there is a clear need for a tensor framework that can explicitly characterize a subspace-consistent dominant component along a prescribed mode together with its orthogonal residual counterpart.
	
	In fact, this question has a classical answer in vector spaces, namely, orthogonal projection theory, which provides a rigorous mechanism for splitting a signal into a component lying in a prescribed subspace and a residual component in its orthogonal complement \cite{kreyszig1991introductory,nimark2012projection}. This type of split is geometrically transparent, energy preserving, and highly interpretable, and thus has long served as a fundamental tool in functional analysis, signal processing, and numerical optimization \cite{Horn2012MatrixAnalysis,chang2020orthogonal}. Naturally, we would like to extend this classical viewpoint to tensor spaces, so as to explicitly split high-dimensional tensor data, along a prescribed mode, into a subspace-consistent dominant component and its orthogonal residual counterpart. This motivates the development of an orthogonal subspace splitting theory for tensors.
	
	Therefore, this paper proposes a novel theoretical framework termed Tensor Orthogonal Subspace Split (TOSS). By inducing orthogonal projections along a prescribed tensor mode, TOSS extends the classical idea of orthogonal subspace splitting from vectors to high-order tensors in a consistent and fiber-wise manner. The proposed framework enables an exact and energy-preserving split of tensor data into a dominant subspace component and an orthogonal residual component. Beyond its theoretical significance, this splitting mechanism offers a new modeling paradigm for challenging high-dimensional tensor data processing tasks, in which the two components can be characterized separately yet jointly exploited, thus potentially breaking through the performance limitations of existing methods.
	
	Within the general TOSS framework, we further investigate an important special case referred to as rank-one TOSS, where the underlying subspace reduces to a single dominant direction. This formulation induces a separable rank-one structure along the splitting mode, yielding an explicit representation that naturally captures consistent patterns while isolating orthogonal deviations. Thanks to its clear geometric interpretation, rank-one TOSS is particularly suitable for data with strong mode-wise consistency. Based on this theoretical tool, hyperspectral image restoration and color video background modeling are considered as two representative tasks, and corresponding optimization models are formulated. Efficient algorithms based on the scaled alternating direction method of multipliers (ADMM) \cite{neal2011distributed} are further developed to solve the resulting problems.
	
	In summary, the main contributions of this paper are as follows:
	\begin{itemize}
		\item A general theoretical framework for tensor orthogonal subspace split, termed TOSS, is established, together with its fundamental properties. The proposed framework provides a new modeling mechanism for challenging high-dimensional tensor data processing tasks, in which the dominant subspace component and the orthogonal residual component can be modeled separately yet jointly exploited.
		\item Rank-one TOSS is identified as an important and practically meaningful special case. It not only provides explicit formulations, clear geometric interpretations, and structural insights that help better understand the general TOSS framework, but also enjoys direct practical relevance, since many types of high-dimensional tensor data naturally admit rank-one TOSS structures.
		\item Based on rank-one TOSS, corresponding structured optimization models are formulated for two representative applications, namely hyperspectral image restoration and color video background modeling. Efficient algorithms are developed for the resulting problems, and extensive experiments demonstrate the effectiveness and superiority of the proposed methods.
	\end{itemize}
	
	The remainder of this paper is organized as follows. Section \ref{sec:2} introduces the notation and preliminaries. Section \ref{sec:3} presents the fundamental theoretical framework of TOSS together with its special case, rank-one TOSS. Based on this framework, Section \ref{sec:4} develops two representative applications, namely hyperspectral image restoration and color video background modeling. Experimental results are provided in Section \ref{sec:5}. Finally, Section \ref{sec:6} concludes the paper. All algorithmic derivations are placed in the \textit{Supplementary Materials}.
	
	
	\section{Notation and Preliminaries}\label{sec:2}
	Throughout the paper, scalars are denoted by lowercase letters (e.g., $x$), vectors by bold lowercase letters (e.g., $\mathbf{x}$), matrices by bold uppercase letters (e.g., $\mathbf{X}$), and tensors by calligraphic letters (e.g., $\mathcal{X}$). $\mathbf{I}$ and $\mathbf{I}_n$ denote the identity matrix of appropriate size and of size $n\times n$, respectively. $(\cdot)^\top$ denotes the transpose of a matrix or vector. $\nabla$ denotes the discrete gradient operator, and $\nabla^\top$ denotes its adjoint operator. For a matrix $\mathbf{X}$ or a tensor $\mathcal{X}$, $\|\cdot\|_F$ denotes the Frobenius norm, and $\|\cdot\|_1$ denotes the entrywise $\ell_1$ norm.

	Let $\mathcal{X}\in\mathbb{R}^{I_1\times I_2\times \cdots \times I_N}$ be an $N$th-order real-valued tensor. 
	For a fixed mode $n\in\{1,\ldots,N\}$, a mode-$n$ fiber of $\mathcal{X}$ is defined as \cite{kolda2009tensor}
	\begin{equation*}\small
		\begin{split}
			&\mathbf{x}_{i_1,\ldots,i_{n-1},i_{n+1},\ldots,i_N}\\
			&=
			\mathcal{X}(i_1,\ldots,i_{n-1},:\,,i_{n+1},\ldots,i_N)\in \mathbb{R}^{I_n},
		\end{split}
	\end{equation*}
	where all indices except the $n$th one are fixed. Given a matrix $\mathbf{A}\in\mathbb{R}^{J\times I_n}$, the mode-$n$ product 
	$\mathcal{Y}:=\mathcal{X}\times_n \mathbf{A}$ 
	is defined by left-multiplying each mode-$n$ fiber of $\mathcal{X}$ \cite{kolda2009tensor}, \emph{i.e.},
	\begin{equation*}\small
		\begin{split}
			&\mathcal{Y}(i_1,\ldots,i_{n-1},:\,,i_{n+1},\ldots,i_N)\\
			&=
			\mathbf{A}\,
			\mathcal{X}(i_1,\ldots,i_{n-1},:\,,i_{n+1},\ldots,i_N).
		\end{split}
	\end{equation*}
	The inner product between two tensors $\mathcal{X}$ and $\mathcal{Y}$ is defined as
	\begin{equation*}\small
		\langle \mathcal{X},\mathcal{Y}\rangle
		:=
		\sum_{i_1,\ldots,i_N}\mathcal{X}(i_1,\ldots,i_N)\mathcal{Y}(i_1,\ldots,i_N),
	\end{equation*}
	and $\|\mathcal{X}\|_F:=\sqrt{\langle \mathcal{X},\mathcal{X}\rangle}$ denotes the Frobenius norm of $\mathcal{X}$.
	
	
	Given a third-order tensor $\mathcal{X}\in\mathbb{R}^{I_1\times I_2\times I_3}$ and $h=\min(I_1,I_2)$, 
	the Logarithmic Schatten-$p$ norm of $\mathcal{X}$ under the t-SVD framework \cite{kilmer2013third}, denoted as 
	$\|\mathcal{X}\|_{L,S_p}$ with $0<p\leq 1$, is defined as \cite{guo2022logarithmic}
	\begin{equation}\small
		\begin{split}
			\|\mathcal{X}\|_{L,S_p}
			:=&
			\frac{1}{I_3}\sum_{k=1}^{I_3}\|\mathbf{X}_f^{(k)}\|_{L,S_p}\\
			=&
			\frac{1}{I_3}\sum_{k=1}^{I_3}\sum_{i=1}^{h}
			\log\big(1+(\mathbf{S}_f^{(k)}(i,i))^p\big),
		\end{split}
		\label{eq:tls_p}
	\end{equation}
	where $\mathbf{X}_f^{(k)}$ denotes the $k$th frontal slice of $\mathcal{X}$ in the Fourier domain,  and $
	\mathbf{X}_f^{(k)}=\mathbf{U}_f^{(k)}\mathbf{S}_f^{(k)}\bigl(\mathbf{V}_f^{(k)}\bigr)^{\top}
	$
	is the singular value decomposition of $\mathbf{X}_f^{(k)}$, with $\mathbf{S}_f^{(k)}$ being a diagonal matrix whose diagonal entries $\mathbf{S}_f^{(k)}(i,i)$ are the singular values of $\mathbf{X}_f^{(k)}$. 
	
	The Logarithmic Schatten-$p$ norm provides a tighter nonconvex surrogate to the tensor rank than the tensor nuclear norm (TNN) \cite{lu2019tensor, zhou2019tensor} by penalizing singular values in a nonlinear and non-uniform manner. In particular, it suppresses small singular values more effectively while preserving dominant ones, thereby reducing information loss \cite{guo2022logarithmic}.
	
	\begin{definition}(Definitions of Several Total Variation (TV) Regularizers)
		For a matrix $\mathbf{A}\in\mathbb{R}^{H\times W}$,
		the spatial TV regularizer $\mathrm{TV}_{s}(\mathbf{A})$ is defined by
		\begin{equation}\small\label{tv1}
			\begin{split}
				\mathrm{TV}_{s}(\mathbf{A})&:=\|\nabla\mathbf{A}\|_1\\
				&:=\|\mathbf{D}_h\mathbf{A}\|_1+\|\mathbf{D}_w\mathbf{A}\|_1,
			\end{split}
		\end{equation}
		where $\mathbf{D}_h$ and $\mathbf{D}_w$ denote the forward finite difference operators along the horizontal and vertical directions, respectively. For a third-order tensor $\mathcal{A}\in\mathbb{R}^{H\times W \times 3}$ representing a color image, the spatial TV regularizer $\mathrm{TV}_{s}(\mathcal{A})$ is defined by
		\begin{equation}\small\label{tv2}
			\begin{split}
				\mathrm{TV}_{s}(\mathcal{A})&:=\|\nabla\mathcal{A}\|_1\\
				&:=\|\mathbf{D}_h\mathcal{A}\|_1+\|\mathbf{D}_w\mathcal{A}\|_1.
			\end{split}
		\end{equation}
		For a third-order tensor $\mathcal{A}\in\mathbb{R}^{H\times W \times B}$ representing a hyperspectral image, the spatial--spectral TV regularizer $\mathrm{TV}_{s,s}(\mathcal{A})$ is defined by
		\begin{equation}\small\label{tv3}
			\begin{split}
				\mathrm{TV}_{s,s}(\mathcal{A})&:=\|\nabla\mathcal{A}\|_1\\
				&:=
				\|\mathbf{D}_h\mathcal{A}\|_1
				+\|\mathbf{D}_w\mathcal{A}\|_1
				+\|\mathbf{D}_b\mathcal{A}\|_1,
			\end{split}
		\end{equation}
		where $\mathbf{D}_b$ denotes the forward finite difference operators along the spectral direction.
		For a fourth-order tensor $\mathcal{A}\in\mathbb{R}^{H\times W \times 3\times T}$ representing a color video, the spatial--temporal TV regularizer $\mathrm{TV}_{s,t}(\mathcal{A})$ is defined by
		\begin{equation}\small\label{tv4}
			\begin{split}
				\mathrm{TV}_{s,t}(\mathcal{A})&:=\|\nabla\mathcal{A}\|_1\\
				&:=
				\|\mathbf{D}_h\mathcal{A}\|_1
				+\|\mathbf{D}_w\mathcal{A}\|_1
				+\|\mathbf{D}_t\mathcal{A}\|_1,
			\end{split}
		\end{equation}
		where $\mathbf{D}_t$ denotes the forward finite difference operators along the temporal direction.
	\end{definition}
	
	For other fundamental concepts and preliminaries related to tensors, the reader is referred to \cite{kolda2009tensor,kilmer2011factorization}.
	\section{Tensor Orthogonal Subspace Split}\label{sec:3}
	\begin{figure*}[htbp]
		\centering
		\includegraphics[width=14.2cm, height=3.2cm]{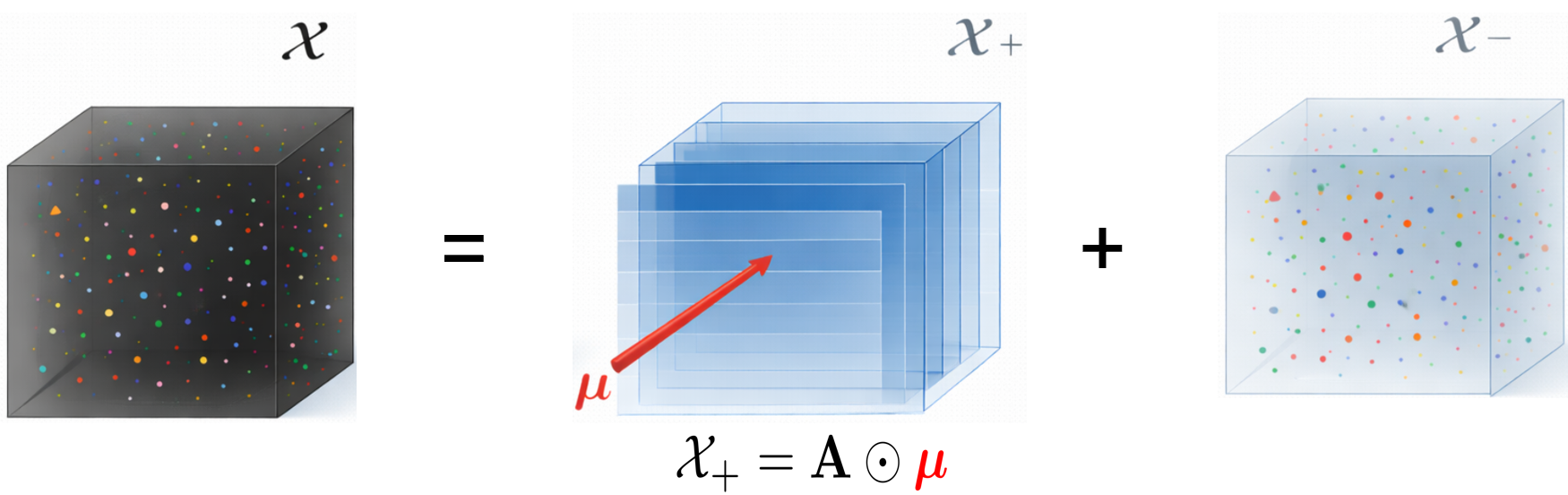}
		\caption{Illustration of the rank-one TOSS for a third-order tensor along the third mode. 
			The original tensor $\mathcal{X}$ is split into a dominant component $\mathcal{X}_+$ and a residual component $\mathcal{X}_-$ such that $\mathcal{X}=\mathcal{X}_+ + \mathcal{X}_-$. 
			The component $\mathcal{X}_+$ exhibits a rank-one structure along the third mode, where all mode-3 fibers are collinear and aligned with a common direction $\boldsymbol{\mu}$, \emph{i.e.}, 
			$\mathcal{X}_+(i,j,:) = a(i,j)\boldsymbol{\mu}$ for some coefficient matrix $\mathbf{A}=[a(i,j)]$. 
			This results in a stack of scaled frontal slices sharing the same spectral/temporal profile. 
			The residual component $\mathcal{X}_-$ contains variations orthogonal to $\boldsymbol{\mu}$, capturing localized perturbations, noise, and fine-scale structures that cannot be represented by the rank-one subspace.}
		\label{rank_one_TOSS_exam}
	\end{figure*}
	\subsection{General Subspace-Induced Tensor Orthogonal Split}
	Let $\mathfrak{U}\subset\mathbb{R}^{I_n}$ be an $r$-dimensional linear subspace with orthogonal complement 
	$\mathfrak{U}^\perp$, and let $\mathbf{U}\in\mathbb{R}^{I_n\times r}$ collect an orthonormal basis of $\mathfrak{U}$,
	\emph{i.e.}, $\operatorname{col}(\mathbf{U})=\mathfrak{U}$ with
	$\mathbf{U}^\top \mathbf{U} = \mathbf{I}_r$.
	
	The orthogonal projection matrix onto $\mathfrak{U}$ is given by
	\begin{equation}\small\label{definP}
		\mathbf{P} := \mathbf{U}\mathbf{U}^\top,
	\end{equation}
	which clearly satisfies $\mathbf{P}^2=\mathbf{P}$ and $\mathbf{P}^\top=\mathbf{P}$.
	\begin{proposition}[Orthogonal Subspace Split]
		For any vector $\mathbf{x}\in\mathbb{R}^{I_n}$, an orthogonal split (also referred to as an orthogonal decomposition in the sense of Hilbert spaces) holds:
		\begin{equation}\small
			\mathbf{x}
			=
			\underbrace{\mathbf{P}\mathbf{x}}_{\mathbf{x}_+}
			+
			\underbrace{(\mathbf{I}-\mathbf{P})\mathbf{x}}_{\mathbf{x}_-},
		\end{equation}
		where $\mathbf{x}_+\in\mathfrak{U}$, $\mathbf{x}_-\in\mathfrak{U}^\perp$, and
		$\langle \mathbf x_+, \mathbf x_- \rangle=0$.
	\end{proposition}
	\begin{proof} 
		Since $\mathbf P$ is an orthogonal projector, we have
		$\mathbf P(\mathbf I-\mathbf P)=\mathbf 0$, which implies
		$\langle \mathbf x_+, \mathbf x_- \rangle
		= \mathbf x^\top \mathbf P(\mathbf I-\mathbf P)\mathbf x = 0$. Actually, such an orthogonal split follows directly from the classical orthogonal projection theorem in linear algebra \cite{kreyszig1991introductory,Horn2012MatrixAnalysis}.
	\end{proof}
	
	The above orthogonal subspace split is defined for vectors in $\mathbb{R}^{I_n}$.
	When a tensor $\mathcal{X} \in \mathbb{R}^{I_1 \times \cdots \times I_N}$ is considered, each mode-$n$ fiber of $\mathcal{X}$ can be naturally viewed as a vector in $\mathbb{R}^{I_n}$. Applying the same orthogonal projector $\mathbf{P}_{n}$ to all mode-$n$ fibers therefore induces a consistent, fiber-wise orthogonal subspace split of the entire tensor. In other words, the vector-level subspace-induced split can be lifted to the tensor domain by acting independently yet coherently along a prescribed mode. This observation motivates the following definition of the
	\emph{Tensor Orthogonal Subspace Split (TOSS)}.
	
	\begin{definition}[TOSS]\label{toss}
		Given a tensor $\mathcal{X}$ and a subspace $\mathfrak{U}_n$ associated with a prescribed mode (e.g., mode-$n$),
		the \emph{TOSS} of $\mathcal{X}$ is defined as
		\begin{equation}\small
			\mathcal{X}_+ := \mathcal{X}\times_n \mathbf{P}_n, \qquad
			\mathcal{X}_- := \mathcal{X}\times_n (\mathbf{I}-\mathbf{P}_n),
		\end{equation}
		where $\mathbf{P}_n$ is the orthogonal projection matrix onto $\mathfrak{U}_n$.
	\end{definition}
	It follows immediately that
	\begin{equation}\small\label{split1}
		\mathcal{X} = \mathcal{X}_+ + \mathcal{X}_-, 
	\end{equation}
	which we refer to as the exact reconstruction property of TOSS.
	\begin{remark}
		The prescribed mode $n$ is determined by the organization of the tensor data and the prior knowledge of the underlying structure. In general, $n$ should be chosen as the mode along which the fibers are expected to share a common dominant structural pattern. For example, for an HSI represented as $\mathcal{Y}\in\mathbb{R}^{H\times W\times B}$, where each frontal slice corresponds to one spectral band, the spectral mode is selected, i.e., $n=3$, since different bands usually share a common spatial structure. For a color video represented as $\mathcal{Y}\in\mathbb{R}^{H\times W\times 3\times T}$, the temporal mode is selected, i.e., $n=4$, because the background is expected to be temporally consistent across frames.
	\end{remark}

	While the subspace-induced split in (\ref{split1}) is formulated in terms of
	orthogonal projections, it is often beneficial to reinterpret such a
	split from an operator-theoretic perspective.
	To this end, following the standard relation
	between an orthogonal projection and the associated reflection
	operator \cite{strang2022introduction}, we define
	\begin{equation}\small\label{definR}
		\mathbf{R}_{n}:=2\mathbf{P}_{n}-\mathbf{I},
	\end{equation}
	and recall the following standard proposition.
	\begin{proposition}\cite{strang2022introduction}\label{prop2}
		The matrix $\mathbf{R}_n$ is a symmetric involution, i.e.,
		\begin{equation}\small
			\mathbf{R}_{n}^\top = \mathbf{R}_{n}, \qquad \mathbf{R}_{n}^2 = \mathbf{I}.
		\end{equation}
		Moreover, $\mathbf{R}_{n}$ has eigenvalue $1$ on $\mathfrak{U}_{n}$ and $-1$ on $\mathfrak{U}_{n}^\perp$.
	\end{proposition}
	
	The proposition \ref{prop2} shows that $\mathbf{R}_{n}$ acts as a reflection operator with respect to the subspace $\mathfrak{U}_{n}$: it leaves fibers in $\mathfrak{U}_{n}$ unchanged and flips the sign of fibers in $\mathfrak{U}_{n}^{\perp}$. Hence, $\mathbf{R}_{n}$ provides a spectral characterization of the orthogonal split between the structured dominant component and its complementary residual part.
	
	From \eqref{definR}, noting that
	\begin{equation}\small
		\mathbf{P}_{n} = \frac{1}{2}(\mathbf{I}+\mathbf{R}_{n}), \qquad
		\mathbf{I}-\mathbf{P}_{n} = \frac{1}{2}(\mathbf{I}-\mathbf{R}_{n}),
	\end{equation}
	the TOSS in Definition \ref{toss} admits the following equivalent representation:
	\begin{equation}\label{pm1}\small
		\small
		\left\{
		\begin{aligned}
			\mathcal{X}_{+}
			&=
			\frac{1}{2}
			\left(
			\mathcal{X}
			+
			\mathcal{X}\times_{n}\mathbf{R}_{n}
			\right), \\
			\mathcal{X}_{-}
			&=
			\frac{1}{2}
			\left(
			\mathcal{X}
			-
			\mathcal{X}\times_{n}\mathbf{R}_{n}
			\right).
		\end{aligned}
		\right.
	\end{equation}
	
	Besides the exact reconstruction property in \eqref{split1}, TOSS possesses the following basic properties as well.
	
	\begin{property}[Fiber-wise Orthogonality]\label{Pro1}
		Define the mode-$n$ fiber-wise inner-product tensor
		\begin{equation*}\small
			\langle \mathcal{X}_{+},\mathcal{X}_{-}\rangle_{(n)}
			\in
			\mathbb{R}^{I_1\times\cdots\times I_{n-1}\times I_{n+1}\times\cdots\times I_N}
		\end{equation*}
		by
		\begin{equation*}\small
			\begin{aligned}
				&\bigl(\langle \mathcal{X}_{+},\mathcal{X}_{-}\rangle_{(n)}\bigr)_{i_1,\ldots,i_{n-1},i_{n+1},\ldots,i_N}\\
				&=
				\bigl\langle
				\mathcal{X}_{+}(i_1,\ldots,i_{n-1},:,i_{n+1},\ldots,i_N),\\
				&\qquad
				\mathcal{X}_{-}(i_1,\ldots,i_{n-1},:,i_{n+1},\ldots,i_N)
				\bigr\rangle .
			\end{aligned}
		\end{equation*}
		Then
		\begin{equation}\small
			\langle \mathcal{X}_{+},\mathcal{X}_{-}\rangle_{(n)}=\mathbf{0}.
			\label{eq:fiberwise_orthogonality_tensor}
		\end{equation}
		That is, for every fixed $(i_1,\ldots,i_{n-1},i_{n+1},\ldots,i_N)$, the corresponding mode-$n$ fibers of $\mathcal{X}_{+}$ and $\mathcal{X}_{-}$ are orthogonal.
	\end{property}
	
	\begin{property}[Energy Preservation]
		The Frobenius norm of $\mathcal{X}$ satisfies
		\begin{equation}\small
			\|\mathcal{X}\|_F^2
			=
			\|\mathcal{X}_+\|_F^2
			+
			\|\mathcal{X}_-\|_F^2.
		\end{equation}
	\end{property}
	
	\begin{proof}
		By construction, $\mathcal{X}=\mathcal{X}_{+}+\mathcal{X}_{-}$. Hence,
		\begin{equation*}\small
			\|\mathcal{X}\|_F^2
			=
			\|\mathcal{X}_{+}+\mathcal{X}_{-}\|_F^2
			=
			\|\mathcal{X}_{+}\|_F^2+\|\mathcal{X}_{-}\|_F^2
			+2\langle \mathcal{X}_{+},\mathcal{X}_{-}\rangle .
		\end{equation*}
		Moreover, the Frobenius inner product $\langle \mathcal{X}_{+},\mathcal{X}_{-}\rangle$
		is the sum of the entries of the mode-$n$ fiber-wise inner-product tensor
		$\langle \mathcal{X}_{+},\mathcal{X}_{-}\rangle_{(n)}$. By  Property~\ref{Pro1}, $
		\langle \mathcal{X}_{+},\mathcal{X}_{-}\rangle_{(n)}=\mathbf{0}$,
		and therefore $\langle \mathcal{X}_{+},\mathcal{X}_{-}\rangle=0$. Substituting this into the above identity yields the desired result.
	\end{proof}
	
	\begin{property}[Invariant/Anti-Invariant under $\mathbf{R}_n$]\label{propo:invar}
		The two components satisfy
		\begin{equation}\small
			\mathcal{X}_+\times_n \mathbf{R}_n = \mathcal{X}_+,
			\qquad
			\mathcal{X}_-\times_n \mathbf{R}_n = -\mathcal{X}_-.
		\end{equation}
	\end{property}
	\begin{proof}
		Using associativity of the mode-$n$ product and the involutive property
		$\mathbf{R}^2=\mathbf{I}$, we have
		\begin{equation*}\small
			(\mathcal{X}\times_n\mathbf{R}_n)\times_n\mathbf{R}_{n}
			= \mathcal{X}\times_n(\mathbf{R}_{n}^2)
			= \mathcal{X}.
		\end{equation*}
		Then
		\begin{equation*}\small
			\begin{split}
				\mathcal{X}_{+}\times_n\mathbf{R}_n
				&= \tfrac12\big(\mathcal{X}+\mathcal{X}\times_n\mathbf{R}_n\big)\times_n\mathbf{R}_n \\
				&= \tfrac12\big(\mathcal{X}\times_n\mathbf{R}_n
				+(\mathcal{X}\times_n\mathbf{R}_n)\times_n\mathbf{R}_n\big) \\
				&= \tfrac12\big(\mathcal{X}\times_n\mathbf{R}_n+\mathcal{X}\big)
				= \mathcal{X}_{+}.
			\end{split}
		\end{equation*}
		
		Similarly,
		\begin{equation*}\small
			\begin{split}
				\mathcal{X}_{-}\times_n\mathbf{R}_n
				&= \tfrac12\big(\mathcal{X}-\mathcal{X}\times_n\mathbf{R}_n\big)\times_n\mathbf{R}_n \\
				&= \tfrac12\big(\mathcal{X}\times_n\mathbf{R}_n
				-(\mathcal{X}\times_n\mathbf{R}_n)\times_n\mathbf{R}_n\big) \\
				&= \tfrac12\big(\mathcal{X}\times_n\mathbf{R}_n-\mathcal{X}\big)
				= -\mathcal{X}_{-}.
			\end{split}
		\end{equation*}
		This completes the proof.
	\end{proof}
	
	\begin{remark}
		The invariance relations in Property \ref{propo:invar} indicate that the TOSS is induced by the involutive operator $\mathbf{R}_n$ along the mode-$n$ fibers. 
		The component $\mathcal{X}_+$ lies in the $+1$ eigenspace of $\mathbf{R}_n$ and characterizes dominant structures preserved under reflection, whereas $\mathcal{X}_-$ belongs to the $-1$ eigenspace and represents orthogonal fluctuations. This reveals that TOSS is not merely an algebraic decomposition, but a geometrically meaningful orthogonal subspace split. Such a structure enables a principled separation between the dominant structured component and the orthogonal residual component, thereby offering interpretability and flexibility for subsequent modeling, optimization, and applications.
	\end{remark}
	
	\subsection{Rank-One TOSS Induced by a Single Direction}
	In general, the number of dominant directions in TOSS should be chosen according to the specific application and data structure, and may also be determined in an adaptive manner. In what follows, we focus on an important special case where the underlying mode-$n$ subspace is one-dimensional, namely, it is spanned by a single dominant direction. This special case is referred to as \emph{rank-one TOSS}.
	
	Specifically, let $\boldsymbol{\mu} \in \mathbb{R}^{I_n}$ be a unit vector, \emph{i.e.},
	$\|\boldsymbol{\mu}\|_2 = 1$, and define the associated one-dimensional subspace
	\begin{equation*}\small
		\mathfrak{U}_{n}^{\boldsymbol{\mu}} := \operatorname{span}\{\boldsymbol{\mu}\} \subset \mathbb{R}^{I_n}.
	\end{equation*}
	The term \textit{rank-one} refers to the fact that the subspace $\mathfrak{U}_{n}^{\boldsymbol{\mu}}$ has dimension one, which will further induce a rank-one structure of the resulting tensor component along mode $n$.
	
	A distinctive property of the rank-one TOSS is that the structured dominant component shares
	a common direction along the splitting mode.
	Let $\mathcal{X}\in\mathbb{R}^{I_1\times\cdots\times I_N}$.
	For any fixed index tuple $(i_1,\ldots,i_{n-1},i_{n+1},\ldots,i_N)$, denote the
	corresponding mode-$n$ fiber by $\mathbf{x} := \mathcal{X}(i_1,\ldots,i_{n-1},:\,,i_{n+1},\ldots,i_N)\in\mathbb{R}^{I_n}$.
	The rank-one TOSS splits each fiber as
	\begin{equation}\small
		\mathbf{x}_+ = (\boldsymbol{\mu}^\top \mathbf{x})\,\boldsymbol{\mu},
		\qquad
		\mathbf{x}_- = \mathbf{x} - (\boldsymbol{\mu}^\top \mathbf{x})\,\boldsymbol{\mu}.
	\end{equation}
	As a consequence, all mode-$n$ fibers of $\mathcal{X}_+$ are collinear and lie in
	$\operatorname{span}\{\boldsymbol{\mu}\}$, implying that $\mathcal{X}_+$ admits a
	rank-one structure along mode $n$.
	
	From an operator viewpoint, the above fiber-wise split can be equivalently
	expressed using the orthogonal projector onto $\mathfrak{U}_{n}^{\boldsymbol{\mu}}$, defined by
	\begin{equation*}\small
		\mathbf{P}_{n}^{\boldsymbol{\mu}}:= \boldsymbol{\mu}\boldsymbol{\mu}^\top,
	\end{equation*}
	which satisfies $(\mathbf{P}_{n}^{\boldsymbol{\mu}})^2=\mathbf{P}_{n}^{\boldsymbol{\mu}}$ and $(\mathbf{P}_{n}^{\boldsymbol{\mu}})^\top=\mathbf{P}_{n}^{\boldsymbol{\mu}}$.
	The associated orthogonal involution induced by $\mathfrak{U}_{n}^{\boldsymbol{\mu}}$ is given by
	\begin{equation}\small\label{houslike}
		\mathbf{R}_{n}^{\boldsymbol{\mu}} := 2\mathbf{P}_{n}^{\boldsymbol{\mu}} - \mathbf{I}
		= 2\boldsymbol{\mu}\boldsymbol{\mu}^\top - \mathbf{I}.
	\end{equation}
	Substituting this involution into the general $\pm$ split in Definition \ref{toss} yields the explicit tensor form
	\begin{equation}\small
		\left\{
		\begin{aligned}
			\mathcal{X}_+
			&=
			\frac{1}{2}\left(\mathcal{X} + \mathcal{X}\times_n \mathbf{R}_{n}^{\boldsymbol{\mu}}\right),\\
			\mathcal{X}_-
			&=
			\frac{1}{2}\left(\mathcal{X} - \mathcal{X}\times_n \mathbf{R}_{n}^{\boldsymbol{\mu}}\right).
		\end{aligned}
		\right.
	\end{equation}
	Moreover, the associated involution \eqref{houslike}
	corresponds to a classical \textit{Householder reflection} with respect to the subspace $\mathfrak{U}_{n}^{\boldsymbol{\mu}}$, which naturally induces the symmetric splitting property of TOSS.
	
	\begin{remark}
		The rank-one TOSS admits a clear geometric interpretation. Specifically, $\mathcal{X}_{+}$ captures the component of $\mathcal{X}$ whose mode-$n$ fibers are aligned with the dominant direction $\boldsymbol{\mu}$, which should be appropriately chosen according to the intrinsic structure of the given data, whereas $\mathcal{X}_{-}$ represents the orthogonal residual that accounts for deviations from this direction. When applied fiber-wise to tensors, rank-one TOSS provides a principled and energy-preserving mechanism for separating consistent structures from orthogonal variations along a prescribed mode.
		
		Such a modeling viewpoint naturally arises in a variety of multidimensional data (as shown in Fig. \ref{TOSS_exam}). For a color image, the RGB channels usually exhibit highly consistent spatial structures, sharing common edges, contours, and object shapes; this cross-channel structural consistency is well characterized by $\mathcal{X}_{+}$, while sensor noise, illumination changes, color fluctuations, or compression artifacts are more naturally absorbed into $\mathcal{X}_{-}$. For a hyperspectral image, although the number of spectral bands is much larger, different bands still commonly share similar spatial layouts, such as object boundaries, texture patterns, and region shapes; these consistent spatial structures are captured by $\mathcal{X}_{+}$, whereas band-dependent noise, stripe corruption, or local anomalous responses are described by $\mathcal{X}_{-}$. For a video sequence with a relatively stable background, frames typically maintain strong consistency along the temporal mode, so that the persistent background is represented by $\mathcal{X}_{+}$, while moving foreground objects, transient occlusions, and temporal disturbances are contained in $\mathcal{X}_{-}$.
	\end{remark}

	\section{Application Examples of Rank-One TOSS}\label{sec:4}
	The proposed TOSS provides a general and flexible framework for splitting high-dimensional tensor data into two orthogonal components along a prescribed mode, guided by an underlying subspace structure. 
	This splitting mechanism is particularly suitable for data whose variability along a given mode is dominated by consistent latent factors, while the remaining variations correspond to residual, noise, or localized structures. 
	Such characteristics are commonly observed in hyperspectral imaging, video analysis, medical imaging \cite{sedighin2024tensor,liu2023survey}, and multi-sensor data fusion \cite{khaleghi2013multisensor}, where different physical or semantic components are naturally entangled along specific tensor modes.
	
	Among the resulting special cases, rank-one TOSS is especially well suited to hyperspectral images and videos with relatively stable backgrounds, as these data often exhibit strong coherence governed by a single dominant direction. Accordingly, this section considers hyperspectral image restoration and color video background modeling as two representative applications to verify the effectiveness of rank-one TOSS, for which corresponding optimization models and efficient solution algorithms are developed.

	\subsection{Hyperspectral Image Restoration}
	
	Let $\mathcal{Y}\in\mathbb{R}^{H\times W\times B}$ denote an observed noise hyperspectral image, where $B$ is the number of spectral bands. The observation is modeled as
	\begin{equation}\small
		\mathcal{Y}=\mathcal{X}+\mathcal{N}+\mathcal{S},
		\label{eq:hsi_observation}
	\end{equation}
	where $\mathcal{X}\in\mathbb{R}^{H\times W\times B}$ denotes the underlying clean hyperspectral image, $\mathcal{N}$ represents dense noise corruption, and $\mathcal{S}$ denotes sparse corruption such as impulse noise, stripe artifacts, or dead pixels.
	
	Using the rank-one TOSS along the spectral mode (mode-$3$), we split $\mathcal{X}$ as $\mathcal{X} = \mathcal{X}_+ + \mathcal{X}_-$, where $\mathcal{X}_+$ represents the dominant spectral component whose mode-$3$ fibers are aligned with a single spectral direction, $\mathcal{X}_-$ captures the orthogonal residual variations.
	\begin{remark}
		Under the rank-one TOSS, the dominant component $\mathcal{X}_+$ admits a separable representation along the spectral mode.
		Specifically, for each spatial location $(i,j)$, the spectral fiber satisfies
		\begin{equation*}\small
			\mathcal{X}_+(i,j,:) = a(i,j)\,\boldsymbol{\mu}\footnote{For hyperspectral image denoising, the direction vector $\boldsymbol{\mu}$ is determined in a data-adaptive manner by principal component analysis (PCA). Specifically, the observed HSI is first reshaped along the spectral mode, and the first principal component is used as the dominant spectral direction. The resulting vector is then normalized to obtain $\boldsymbol{\mu}$. This choice reflects the main spectral variation of the given HSI and is fixed throughout the subsequent optimization process.},
		\end{equation*}
		where $a(i,j)=\boldsymbol{\mu}^\top\mathcal{X}(i,j,:)$ denotes the projection coefficient.
		Equivalently, in tensor form,
		\begin{equation*}\small
			\mathcal{X}_+ = \mathcal{M}(\mathbf{A}):=\mathbf{A}\odot\boldsymbol{\mu},
		\end{equation*}
		with $\mathbf{A}\in\mathbb{R}^{H\times W}$ being the spatial coefficient map, where $\odot$ denotes elementwise multiplication after broadcasting $\mathbf{A}$ along the third mode (i.e., if $\boldsymbol{\mu}\in\mathbb{R}^B$, then $\mathcal{X}_+(:,:,b)=\mu_b\,\mathbf{A}$).
	\end{remark}
	Therefore, the observation model (\ref{eq:hsi_observation}) can be rewritten as
	\begin{equation}\small
		\mathcal{Y}=\mathcal{M}(\mathbf{A})+\mathcal{X}_{-}+\mathcal{N}+\mathcal{S}.
	\end{equation}
	
	Under the rank-one TOSS, $\mathbf{A}$ characterizes the principal spatial structure associated with the common spectral direction $\boldsymbol{\mu}$. Since this component is expected to encode the main piecewise-smooth spatial content of the hyperspectral image, including major object regions, boundaries, and slowly varying intensity patterns, a spatial total variation regularizer $\mathrm{TV}_s(\mathbf{A})$ is imposed to promote spatial regularity while preserving essential structural edges. The complementary component $\mathcal{X}_-$ collects the residual spectral--spatial information that cannot be represented by the dominant rank-one consistent part. Rather than being treated as purely random noise, $\mathcal{X}_-$ is regarded as a structured complementary component that may contain fine details, local variations, and non-dominant spectral responses. Owing to the strong correlation and redundancy that still persist across its spatial and spectral dimensions, $\mathcal{X}_-$ is regularized by the logarithmic Schatten-$p$ norm to capture its global low-rank structure, together with the $\mathrm{TV}_{s,s}$ regularizer to enforce local smoothness and piecewise continuity in both the spatial dimensions and the spectral direction. In addition, the sparse corruption term $\mathcal{S}$ is introduced to account for impulsive degradations, stripe noise, and other sparse outliers. Consequently, the following unified hyperspectral image restoration model is proposed:
	
	\begin{equation}\small
		\begin{aligned}
			\min_{\mathbf{A},\,\mathcal{X}_-,\,\mathcal{S}} \quad 
			& \frac{1}{2}\big\|\mathcal{M}(\mathbf{A}) + \mathcal{X}_- + \mathcal{S} - \mathcal{Y} \big\|_F^2  \\
			& + \lambda_1 \, \mathrm{TV}_s(\mathbf{A})
			+ \lambda_2\, \|\mathcal{X}_-\|_{L,S_p}\\
			&+ \lambda_3 \, \mathrm{TV}_{s,s}(\mathcal{X}_-)
			+ \lambda_4 \, \|\mathcal{S}\|_1 ,
		\end{aligned}
		\label{eq:hsi_model}
	\end{equation}
	where $\lambda_1,\lambda_2,\lambda_3,\lambda_4>0$ are regularization parameters, $\|\mathcal{X}_-\|_{L,S_p}$, $\mathrm{TV}_s(\mathbf{A})$, and $\mathrm{TV}_{s,s}(\mathcal{X}_-)$ are defined as in (\ref{eq:tls_p}), (\ref{tv1}), and (\ref{tv3}), respectively.
	
	The derivation of the optimization algorithm for model~\eqref{eq:hsi_model} is presented in Section \uppercase\expandafter{\romannumeral 1} of the  \textit{Supplementary Materials}.
	
	\subsection{Color Video Background Modeling}
	Let $\mathcal{X}\in\mathbb{R}^{H\times W\times 3\times T}$ denote a color video with a stable background, where $T$ represents the number of temporal frames. Using the rank-one TOSS along the temporal mode (mode-$4$), we split $\mathcal{X}$ as $\mathcal{X}=\mathcal{X}_{+}+\mathcal{X}_{-}$, where $\mathcal{X}_{+}$ represents the background component whose mode-$4$ fibers are aligned with a single temporal direction, while $\mathcal{X}_{-}$ captures the foreground objects and other temporally varying residual components.
	\begin{remark}
		Under the rank-one TOSS along the temporal mode, the background component $\mathcal{X}_{+}$ admits a separable representation along the fourth mode. Specifically, for each spatial-color location $(i,j,k)$, the mode-4 fiber satisfies
		\begin{equation*}\small
			\mathcal{X}_{+}(i,j,k,:) = a(i,j,k)\,\boldsymbol{\mu}\footnote{For color video background modeling, we adopt a particularly simple yet effective choice by setting $\boldsymbol{\mu}$ as a uniform vector along the temporal mode. This choice can be interpreted as projecting onto the mean temporal direction of the video frames, which captures the dominant consistent background component shared across frames.},
		\end{equation*}
		where $a(i,j,k)=\boldsymbol{\mu}^{\top}\mathcal{X}(i,j,k,:)$ denotes the projection coefficient. Equivalently, in tensor form,
		\begin{equation*}\small
			\mathcal{X}_{+}=\mathcal{M}(\mathcal{A}) := \mathcal{A}\odot \boldsymbol{\mu},
		\end{equation*}
		with $\mathcal{A}\in\mathbb{R}^{H\times W\times 3}$ being the spatial-color coefficient tensor, where $\odot$ denotes elementwise multiplication after broadcasting $\mathcal{A}$ along the fourth mode (i.e., if $\boldsymbol{\mu}\in\mathbb{R}^{T}$, then $\mathcal{X}_{+}(:,:,:,t)=\mu_t\,\mathcal{A}$). 
	\end{remark}
	
	Although directly applying rank-one TOSS to $\mathcal{X}$ can already produce a relatively satisfactory background estimate for some simple video sequences, as illustrated in the third row of Fig. \ref{TOSS_exam}, such a direct split is still limited in more challenging scenarios. In particular, when foreground motions are more complex or the video contains various temporal fluctuations, the obtained background component may be contaminated by moving objects and residual disturbances. To achieve a more accurate and robust background modeling, it is therefore desirable to further incorporate appropriate structural priors into an optimization framework. Specifically, the background component $\mathcal{M}(\mathcal{A})$ is expected to be spatially smooth and structurally consistent, which motivates the use of the spatial TV regularization $\mathrm{TV}_{s}(\mathcal{A})$. Meanwhile, the residual component $\mathcal{X}_{-}$, mainly corresponding to moving foreground objects and other temporal variations, is often sparse in space while also exhibiting certain spatial-temporal continuity, which can be characterized by the $\ell_{1}$ norm and the spatial-temporal TV regularization $\mathrm{TV}_{s,t}(\mathcal{X}_{-})$. Based on these considerations, we consider the following optimization model:
	\begin{equation}\small\label{eq:video_model}
		\begin{split}
			\min_{\mathcal{A},\,\mathcal{X}_{-}}&
			\frac{1}{2}\left\|\mathcal{M}(\mathcal{A})+\mathcal{X}_{-}-\mathcal{X}\right\|_{F}^{2}
			+\lambda_{1}\,\mathrm{TV}_{s}(\mathcal{A})\\
			&+\lambda_{2}\left\|\mathcal{X}_{-}\right\|_{1}
			+\lambda_{3}\,\mathrm{TV}_{s,t}(\mathcal{X}_{-}),
		\end{split}
	\end{equation}
	where $\lambda_{1},\lambda_{2},\lambda_{3}>0$ are regularization parameters, $\mathrm{TV}_{s}(\mathcal{A})$, and $\mathrm{TV}_{s,t}(\mathcal{X}_-)$ are defined as in (\ref{tv2}) and (\ref{tv4}), respectively.
	
	The derivation of the optimization algorithm for model~\eqref{eq:video_model} is presented in Section \uppercase\expandafter{\romannumeral 2} of the  \textit{Supplementary Materials}.
	
	\section{Experimental Results}\label{sec:5}
	In this section, we thoroughly validate the effectiveness and superiority of the proposed method in two practical real-world applications: hyperspectral image restoration and color video background modeling. All experiments are run in MATLAB 2024b
	under Windows 11 on a personal computer with a 3.90GHz CPU and 64GB of memory.
	
	\subsection{Experiments on Hyperspectral Image Restoration}
	
	To validate the effectiveness of the proposed method, three widely used hyperspectral image (HSI) datasets, namely \textit{Pavia}, \textit{DCMall}, and \textit{RemoteImage}, are adopted in the experiments. For a fair comparison, all datasets are preprocessed so that each data sample has a unified size of $200 \times 200 \times 80$.
	
	To comprehensively evaluate the robustness of different methods, four noise configurations are considered in this study. Specifically, Gaussian noise (denoted as G) and salt-and-pepper noise (denoted as S) are introduced either individually or in combination. The detailed settings are summarized as follows:
	\begin{enumerate}
		\item Case 1 ($\text{G} = 0.1$): Gaussian noise with variance $0.1$ is added to the clean HSI.
		\item Case 2 ($\text{G} = 0.2$): Gaussian noise with variance $0.2$ is added.
		\item Case 3 ($\text{S} = 0.1, \text{G} = 0.01$): A mixed noise case consisting of salt-and-pepper noise (density $0.1$) and low-level Gaussian noise (variance $0.01$).
		\item Case 4 ($\text{S} = 0.1, \text{G} = 0.1$): A more challenging mixed noise scenario combining salt-and-pepper noise (density $0.1$) with moderate Gaussian noise (variance $0.1$).
	\end{enumerate}

	Before conducting the detailed comparison experiments, we first visualize the effect of the proposed rank-one TOSS mechanism under two representative degradation settings, as shown in Fig. \ref{fig.rankone_combined}. For the clean image without additive corruption, $\mathcal{X}_{+}$ captures the dominant structurally consistent information of the hyperspectral image, such as the main contours, boundaries, and spatial layouts, while $\mathcal{X}_{-}$ contains local details orthogonal to this dominant structure. When the image is corrupted by additive noise, $\mathcal{X}_{+}$ still preserves the principal structural information and is barely affected by noise, whereas most noise perturbations are absorbed into the residual component $\mathcal{X}_{-}$. This behavior agrees well with the theoretical motivation of rank-one TOSS. After restoration, the noise in $\mathcal{X}_{-}$ is effectively suppressed, leaving mainly local detail information complementary to $\mathcal{X}_{+}$. The final recovered image is then obtained by combining these two components, thereby achieving noise removal while preserving both dominant structures and meaningful local details.
	
	Four quantitative metrics, including PSNR, SSIM\cite{1284395}, FSIM\cite{5705575}, and ERGAS\cite{wald2002data}, are employed to comprehensively assess the reconstruction performance. Specifically, PSNR and SSIM are classical spatial-domain metrics that evaluate pixel-wise fidelity and structural similarity, respectively. FSIM further measures feature similarity by incorporating phase congruency and gradient information. In contrast, ERGAS is a spectral-domain metric tailored for hyperspectral images, which reflects the relative global reconstruction error across spectral bands. In general, higher values of PSNR, SSIM, and FSIM, along with a lower ERGAS value, indicate better reconstruction performance. The proposed method is compared with several representative approaches based on low-rank and sparse modeling, including SNN\cite{huang2015provable}, TNN\cite{lu2019tensor}, LRTV\cite{he2015total}, TLR-SSTV\cite{chen2018tensor}, TCTV\cite{WangHailin2023}, and SSCTV-RPCA\cite{GAO2025128885}. For our proposed method, the parameters are set as follows. Under Case 3, we set $\lambda_2 = 100$; for the other three noise scenarios, $\lambda_2 = 300$. The remaining parameters are fixed: $\lambda_1 = 0.01$, $\lambda_3 = 0.001$ and $\lambda_4 = 0.1$. For a fair comparison, the parameters of all competing methods are first set according to the corresponding original papers and then carefully fine-tuned to achieve their optimal performance on each dataset.
	
	The denoising performance metrics under different levels of Gaussian noise and under mixed noise scenarios (salt-and-pepper noise of density 0.1 combined with Gaussian noise) are presented in Tables \ref{table:hsirecover1} and \ref{table:hsirecover2}, respectively. Under pure Gaussian noise conditions ($\text{G}=0.1$ and $\text{G}=0.2$), the proposed method consistently outperforms all competing methods across all three datasets in terms of PSNR, SSIM, FSIM, and ERGAS. Notably, even when the noise level increases to $\text{G}=0.2$, the proposed method maintains the best performance, demonstrating strong robustness.
	
	Under the more challenging mixed-noise scenarios ($\text{S}=0.1, \text{G}=0.01$ and $\text{S}=0.1, \text{G}=0.1$), the advantages of our method become even more pronounced. For instance, on the RemoteImage dataset under the mixed noise of $\text{S}=0.1$ and $\text{G}=0.01$, our method outperforms the second-best method (TCTV) by $2.118$ dB in PSNR and reduces ERGAS by $19.83\%$. These results clearly indicate that our method effectively suppresses both sparse and Gaussian noise simultaneously, demonstrating a particular strength in preserving spectral fidelity.
	
	For visual comparison, we select multiple bands from each HSI to demonstrate the restoration performance of all competing methods. Specifically, Figs. \ref{fig.3} and \ref{fig.4} present the restored images under various noise conditions for the Pavia and DCMall datasets, respectively (showing four representative bands per dataset). Additionally, Fig. \ref{fig.5} shows the recovered RemoteImage HSI (band 25) obtained by all methods under the mixed-noise condition of $\text{S}=0.1$ and $\text{G}=0.01$. From these figures, it can be clearly observed that our proposed method achieves superior denoising performance. Specifically, compared with other competing methods, our results preserve sharper edges and more faithful local details. Moreover, TNN, LRTV, and TLR-SSTV suffer from evident color distortion under Gaussian noise, whereas our method successfully mitigates this problem, yielding higher color fidelity and further validating its superiority in spectral consistency preservation. 
	
	Notably, several competing methods compared above, including LRTV, TLR-SSTV, TCTV and SSCTV-RPCA, are also built upon total variation regularization. Specifically, LRTV and TLR-SSTV employ TV as a separate additive term to enforce local smoothness alongside low-rank priors, while TCTV and SSCTV-RPCA integrate TV-type regularization into a unified regularizer that jointly encodes low-rankness and local continuity. Despite their different implementations of TV, none of these methods introduce an explicit orthogonal splitting mechanism before applying TV constraints. Our method, which builds upon the proposed TOSS framework, consistently outperforms these TV-based baselines by a notable margin in both quantitative metrics and visual quality. This finding suggests that the performance gain cannot be attributed solely to the use of TV regularization. Rather, the advantage comes from our orthogonal splitting strategy, which decomposes the tensor into complementary components and allows TV to act on each component without mutual interference. This further demonstrates that the superiority of the proposed method stems primarily from the principled orthogonal splitting mechanism, rather than merely from the incorporation of TV regularization.
	
	Based on the above quantitative and visual comparisons, we further analyze the computational efficiency of the proposed method. Fig. \ref{ave_time} reports the average running time (in seconds) of all competing methods across three HSI datasets under different noise cases. It can be observed that our method achieves a competitive average runtime of 39.67 seconds, which is significantly lower than several representative methods such as SNN (55.38 s), TNN (75.66 s), TCTV (157.93 s), and SSCTV-RPCA (53.29 s). Although LRTV (7.00 s) and TLR-SSTV (20.28 s) are faster, their reconstruction quality is substantially inferior, as evidenced by the quantitative metrics and visual results discussed previously. This indicates that our method strikes a more favorable balance between computational cost and restoration accuracy. This computational advantage stems from the proposed TOSS framework. Unlike existing methods that apply low-rank or TV regularization directly to the original tensor, TOSS explicitly splits the tensor into a dominant component lying in a prescribed subspace and a residual component in its orthogonal complement. This orthogonal splitting allows each component to be optimized independently without mutual interference. This orthogonal splitting strategy leads to smoother convergence and faster per-iteration updates. In contrast, methods like TCTV and SSCTV-RPCA integrate TV into a unified low-rank formulation, resulting in more complex subproblems and slower runtime.
	
	\begin{figure*}[p]
		\centering
		\vspace*{-\topskip}
		\vspace{-1cm}
		
		\begin{minipage}[t]{0.48\textwidth}
			\centering
			\text{(a) Case 1 ($\text{G}=0.1$).}\\[4pt]
			\renewcommand{\arraystretch}{0.5}
			\setlength\tabcolsep{0.5pt}
			\begin{tabular}{ccc}
				\centering
				\includegraphics[width=26mm, height=22mm]{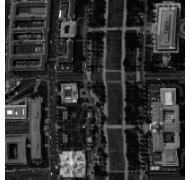} &
				\includegraphics[width=26mm, height=22mm]{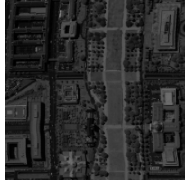}&
				\includegraphics[width=26mm,height=22mm]{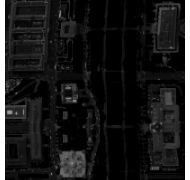} \\
				\scriptsize {DCMall} & \scriptsize $\mathcal{X}_+$ & \scriptsize $\mathcal{X}_-$ \\
				\includegraphics[width=26mm, height=22mm]{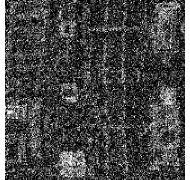} &
				\includegraphics[width=26mm, height=22mm]{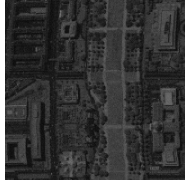} &
				\includegraphics[width=26mm, height=22mm]{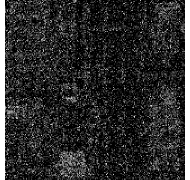} \\
				\scriptsize Noisy & \scriptsize $\mathcal{X}_+$ & \scriptsize $\mathcal{X}_-$ \\
				\includegraphics[width=26mm, height=22mm]{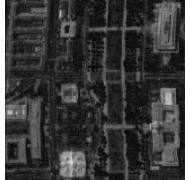} &
				\includegraphics[width=26mm, height=22mm]{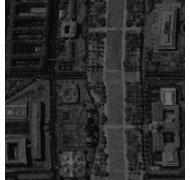} &
				\includegraphics[width=26mm, height=22mm]{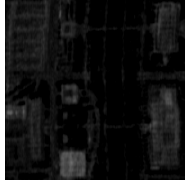} \\
				\scriptsize Recovered & \scriptsize $\mathcal{X}_+$ & \scriptsize $\mathcal{X}_-$ \\
			\end{tabular}
			\label{fig.rankone1}
		\end{minipage}
		\hfill
		\begin{minipage}[t]{0.48\textwidth}
			\centering
			\text{(b) Case 3 ($\text{S}=0.1, \text{G}=0.01$).}\\[4pt]
			\renewcommand{\arraystretch}{0.5}
			\setlength\tabcolsep{0.5pt}
			\begin{tabular}{ccc}
				\centering
				\includegraphics[width=26mm, height=22mm]{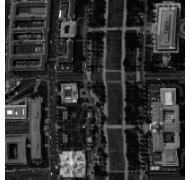} &
				\includegraphics[width=26mm, height=22mm]{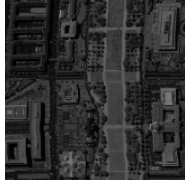} &
				\includegraphics[width=26mm, height=22mm]{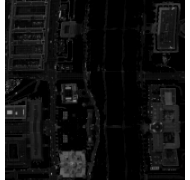} \\
				\scriptsize {DCMall} & \scriptsize $\mathcal{X}_+$ & \scriptsize $\mathcal{X}_-$ \\
				\includegraphics[width=26mm, height=22mm]{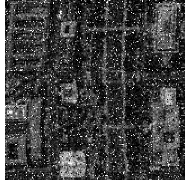} &
				\includegraphics[width=26mm, height=22mm]{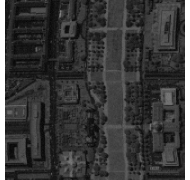} &
				\includegraphics[width=26mm, height=22mm]{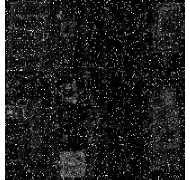} \\
				\scriptsize Noisy & \scriptsize $\mathcal{X}_+$ & \scriptsize $\mathcal{X}_-$ \\
				\includegraphics[width=26mm, height=22mm]{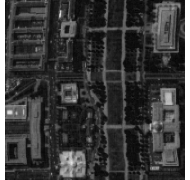} &
				\includegraphics[width=26mm, height=22mm]{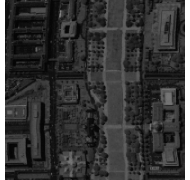} &
				\includegraphics[width=26mm, height=22mm]{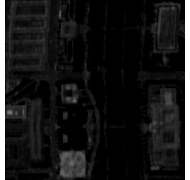} \\
				\scriptsize Recovered & \scriptsize $\mathcal{X}_+$ & \scriptsize $\mathcal{X}_-$ \\
			\end{tabular}
			\label{fig.rankone2}
		\end{minipage}
		
		\vspace{0.5cm}
		\caption{Split results of the DCMall under two noise cases. For each subfigure, the three columns show the original image, the dominant rank-one component ($\mathcal{X}_+$), and the residual component ($\mathcal{X}_-$), respectively. Rows represent the original input, noisy input, and the recovered results by our proposed method.}
		\label{fig.rankone_combined}
	\end{figure*}

	\begin{table*}[tp]
		\renewcommand{\arraystretch}{1.1}
		\setlength\tabcolsep{10pt}
		\footnotesize
		\caption{
			The HSI denoising performances of all competing methods under Gaussian noise with different levels.
		}\label{table:hsirecover1}
		\setlength{\abovecaptionskip}{6pt}
		\setlength{\belowcaptionskip}{6pt}
		\centering
		\begin{tabular}{c||c|c|c|c|c|c|c|c}
			\Xhline{1pt}
			Noise & \multicolumn{4}{c|}{G = 0.1} & \multicolumn{4}{c}{G = 0.2} \\
			\cline{2-9}
			Method & PSNR & SSIM & FSIM & ERGAS & PSNR & SSIM & FSIM & ERGAS \\
			\Xhline{1pt}
			\multicolumn{9}{c}{Pavia}\\
			\hline
			\hline
			SNN\cite{huang2015provable} & 24.014 & 0.742 & 0.877 & 230.348 & 21.997 & 0.636 & 0.840 & 290.705 \\
			TNN\cite{lu2019tensor} & 24.233 & 0.618 & 0.812 & 228.456 & 22.472 & 0.483 & 0.765 & 277.356 \\
			LRTV\cite{he2015total} & 25.239 & 0.657 & 0.775 & 199.480 & 23.877 & 0.570 & 0.722 & 234.342 \\
			TLR-SSTV\cite{chen2018tensor} & 24.776 & 0.612 & 0.786 & 212.024 & 23.613 & 0.535 & 0.749 & 242.248 \\
			TCTV\cite{WangHailin2023} & 26.704 & 0.755 & 0.858 & 171.792 & 24.840 & 0.661 & 0.815 & 211.345 \\
			SSCTV-RPCA\cite{GAO2025128885} & 25.803 & 0.710 & 0.861 & 190.854 & 23.971 & 0.621 & 0.823 & 234.541 \\
			\textbf{Ours} & \textbf{27.236} & \textbf{0.777} & \textbf{0.884} & \textbf{159.089} & \textbf{25.227} & \textbf{0.682} & \textbf{0.846} & \textbf{199.858} \\
			\hline
			\hline
			\multicolumn{9}{c}{DCMall}\\
			\hline
			\hline
			SNN\cite{huang2015provable} & 22.140 & 0.790 & 0.896 & 241.269 & 19.937 & 0.697 & 0.858 & 310.102 \\
			TNN\cite{lu2019tensor} & 22.520 & 0.708 & 0.856 & 234.331 & 20.896 & 0.608 & 0.817 & 280.380 \\
			LRTV\cite{he2015total} & 22.673 & 0.665 & 0.775 & 227.392 & 21.297 & 0.578 & 0.722 & 267.348 \\
			TLR-SSTV\cite{chen2018tensor} & 22.066 & 0.625 & 0.804 & 243.509 & 20.914 & 0.542 & 0.765 & 277.833 \\
			TCTV\cite{WangHailin2023} & 24.776 & 0.801 & 0.892 & 180.865 & 23.154 & 0.731 & 0.861 & 216.757 \\
			SSCTV-RPCA\cite{GAO2025128885} & 24.732 & 0.796 & 0.898 & 183.022 & 22.751 & 0.720 & 0.866 & 228.634 \\
			\textbf{Ours} & \textbf{25.317} & \textbf{0.817} & \textbf{0.903} & \textbf{168.990} & \textbf{23.607} & \textbf{0.750} & \textbf{0.876} & \textbf{204.188} \\
			\hline
			\hline
			\multicolumn{9}{c}{RemoteImage}\\
			\hline
			\hline
			SNN\cite{huang2015provable} & 23.139 & 0.707 & 0.883 & 173.897 & 20.917 & 0.594 & 0.849 & 223.901 \\
			TNN\cite{lu2019tensor} & 23.771 & 0.532 & 0.822 & 166.572 & 22.182 & 0.399 & 0.790 & 197.927 \\
			LRTV\cite{he2015total} & 25.134 & 0.529 & 0.697 & 140.526 & 24.031 & 0.469 & 0.659 & 159.492 \\
			TLR-SSTV\cite{chen2018tensor} & 24.291 & 0.497 & 0.755 & 160.440 & 23.551 & 0.449 & 0.734 & 171.934 \\
			TCTV\cite{WangHailin2023} & 25.600 & 0.654 & 0.860 & 135.957 & 23.790 & 0.544 & 0.823 & 165.872 \\
			SSCTV-RPCA\cite{GAO2025128885} & 25.066 & 0.664 & 0.869 & 142.269 & 23.101 & 0.562 & 0.832 & 176.548 \\
			\textbf{Ours} & \textbf{26.838} & \textbf{0.752} & \textbf{0.897} & \textbf{117.527} & \textbf{24.278} & \textbf{0.626} & \textbf{0.853} & \textbf{154.996} \\
			\hline
			\hline
			\Xhline{1pt}
		\end{tabular}
		\vspace{-0.3cm}
	\end{table*}

	\begin{table*}[tp]
		\renewcommand{\arraystretch}{1.1}
		\setlength\tabcolsep{10pt}
		\footnotesize
		\caption{
			The HSI denoising performances of all competing methods under mixed noise (salt-and-pepper and Gaussian) with different levels.
		}\label{table:hsirecover2}
		\setlength{\abovecaptionskip}{5pt}
		\setlength{\belowcaptionskip}{5pt}
		\centering
		\begin{tabular}{c||c|c|c|c|c|c|c|c}
			\Xhline{1pt}
			Noise & \multicolumn{4}{c|}{S = 0.1, G = 0.01} & \multicolumn{4}{c}{S = 0.1, G = 0.1} \\
			\cline{2-9}
			Method & PSNR & SSIM & FSIM & ERGAS & PSNR & SSIM & FSIM & ERGAS \\
			\Xhline{1pt}
			\multicolumn{9}{c}{Pavia}\\
			\hline
			\hline
			SNN\cite{huang2015provable} & 29.992 & 0.917 & 0.951 & 117.749 & 24.228 & 0.719 & 0.863 & 225.212 \\
			TNN\cite{lu2019tensor} & 29.261 & 0.863 & 0.918 & 135.590 & 23.502 & 0.569 & 0.794 & 248.374 \\
			LRTV\cite{he2015total} & 29.610 & 0.851 & 0.895 & 120.921 & 24.859 & 0.625 & 0.757 & 208.554 \\
			TLR-SSTV\cite{chen2018tensor} & 28.304 & 0.802 & 0.888 & 141.344 & 22.710 & 0.458 & 0.640 & 269.244 \\
			TCTV\cite{WangHailin2023} & 31.615 & 0.908 & 0.940 & 101.874 & 25.717 & 0.715 & 0.840 & 192.512 \\
			SSCTV-RPCA\cite{GAO2025128885} & 30.641 & 0.872 & 0.932 & 109.830 & 24.901 & 0.672 & 0.845 & 211.589 \\
			\textbf{Ours} & \textbf{32.931} & \textbf{0.929} & \textbf{0.959} & \textbf{84.660} & \textbf{26.283} & \textbf{0.750} & \textbf{0.877} & \textbf{177.557} \\
			\hline
			\hline
			\multicolumn{9}{c}{DCMall}\\
			\hline
			\hline
			SNN\cite{huang2015provable} & 28.927 & 0.942 & 0.966 & 111.820 & 21.615 & 0.761 & 0.881 & 256.127 \\
			TNN\cite{lu2019tensor} & 27.352 & 0.887 & 0.936 & 140.890 & 21.770 & 0.668 & 0.840 & 254.953 \\
			LRTV\cite{he2015total} & 27.031 & 0.856 & 0.901 & 137.416 & 22.201 & 0.630 & 0.755 & 239.865 \\
			TLR-SSTV\cite{chen2018tensor} & 26.123 & 0.831 & 0.906 & 153.028 & 20.482 & 0.494 & 0.656 & 292.021 \\
			TCTV\cite{WangHailin2023} & 29.506 & 0.923 & 0.953 & 108.540 & 23.946 & 0.771 & 0.878 & 198.672 \\
			SSCTV-RPCA\cite{GAO2025128885} & 30.149 & 0.922 & 0.955 & 98.687 & 23.936 & 0.769 & 0.886 & 200.362 \\
			\textbf{Ours} & \textbf{30.737} & \textbf{0.941} & \textbf{0.963} & \textbf{92.380} & \textbf{24.529} & \textbf{0.794} & \textbf{0.894} & \textbf{185.114} \\
			\hline
			\hline
			\multicolumn{9}{c}{RemoteImage}\\
			\hline
			\hline
			SNN\cite{huang2015provable} & 29.344 & 0.906 & 0.954 & 86.689 & 22.764 & 0.679 & 0.872 & 181.382 \\
			TNN\cite{lu2019tensor} & 28.426 & 0.821 & 0.911 & 102.510 & 23.271 & 0.489 & 0.809 & 176.378 \\
			LRTV\cite{he2015total} & 28.248 & 0.715 & 0.827 & 97.467 & 24.727 & 0.500 & 0.673 & 147.546 \\
			TLR-SSTV\cite{chen2018tensor} & 26.911 & 0.681 & 0.850 & 120.428 & 22.846 & 0.389 & 0.548 & 186.522 \\
			TCTV\cite{WangHailin2023} & 30.292 & 0.864 & 0.934 & 83.351 & 25.047 & 0.617 & 0.847 & 145.050 \\
			SSCTV-RPCA\cite{GAO2025128885} & 29.831 & 0.846 & 0.932 & 84.529 & 24.464 & 0.629 & 0.857 & 151.880 \\
			\textbf{Ours} & \textbf{32.410} & \textbf{0.920} & \textbf{0.958} & \textbf{66.820} & \textbf{26.218} & \textbf{0.722} & \textbf{0.887} & \textbf{126.197} \\
			\hline
			\hline
			\Xhline{1pt}
		\end{tabular}
		\vspace{-0.3cm}
	\end{table*}
	
	\begin{figure*}[htbp]
		\centering
		\includegraphics[width=14cm, height=8cm]{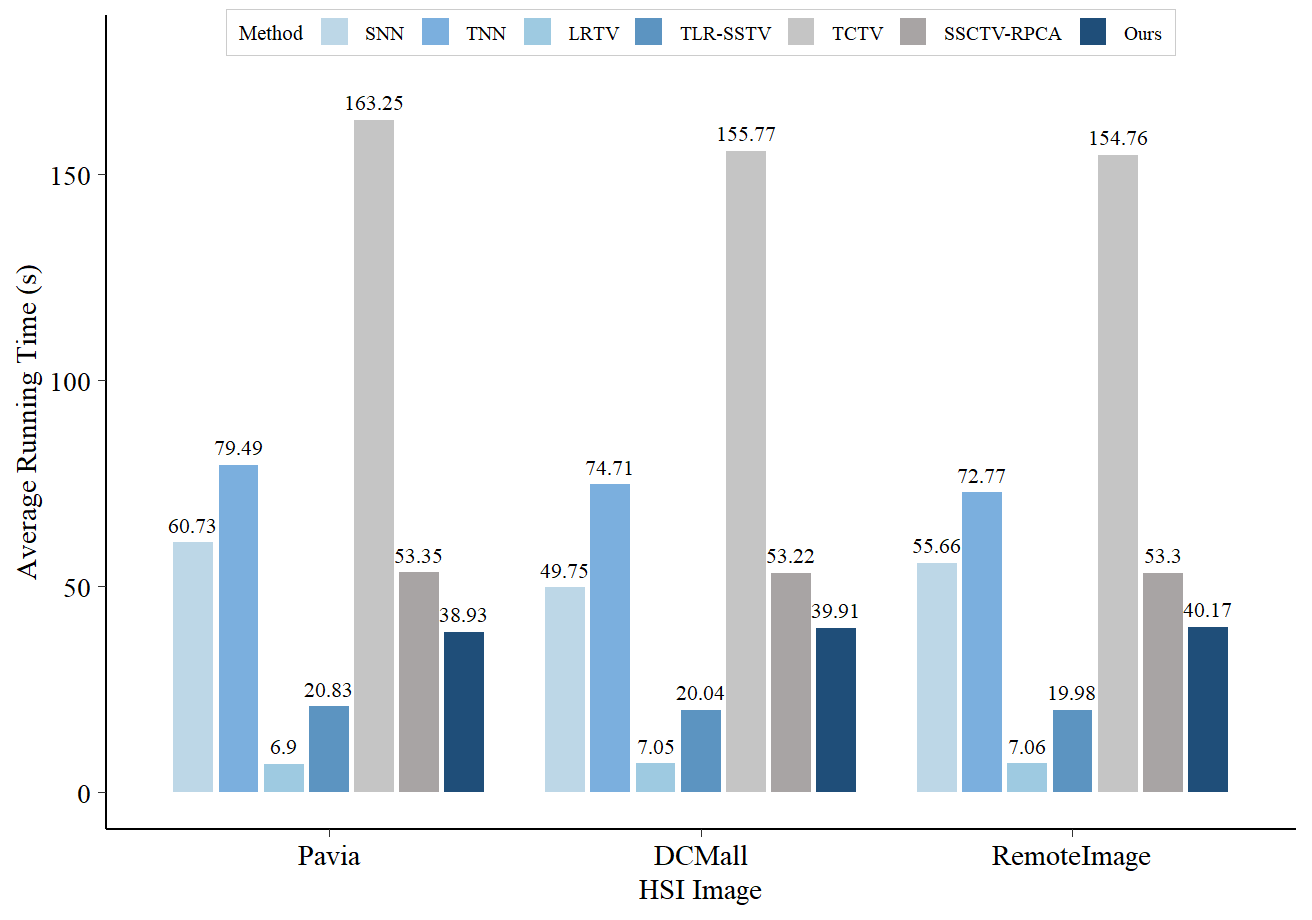}
		\caption{Average running time of all competing methods on three HSI images (Pavia, DCMall, and RemoteImage) under different noise cases.}
		\label{ave_time}
	\end{figure*}

	\begin{figure*}[tp]
		\renewcommand{\arraystretch}{1.2}
		\setlength\tabcolsep{1pt}
		\centering
		\vspace{-0.1cm}
		\begin{tabular}{ccccccccc}
			\centering
			\includegraphics[width=17mm, height = 20mm]{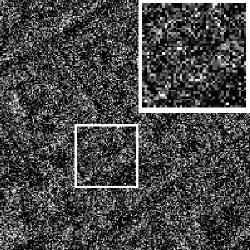}&
			\includegraphics[width=17mm, height = 20mm]{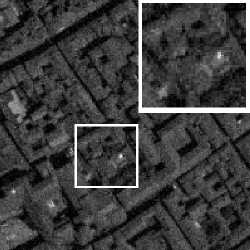}&
			\includegraphics[width=17mm, height = 20mm]{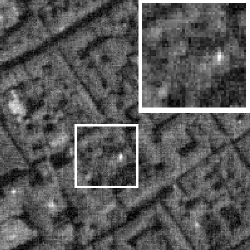}&
			\includegraphics[width=17mm, height = 20mm]{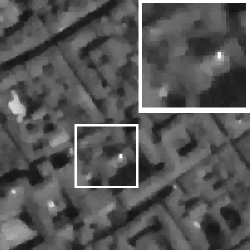}&
			\includegraphics[width=17mm, height = 20mm]{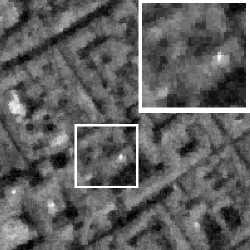}&
			\includegraphics[width=17mm, height = 20mm]{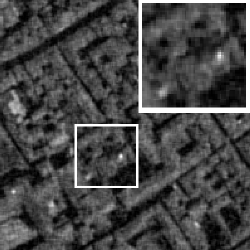}&
			\includegraphics[width=17mm, height = 20mm]{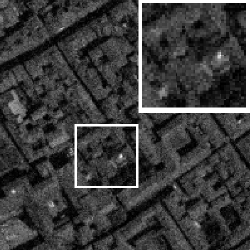}&
			\includegraphics[width=17mm, height = 20mm]{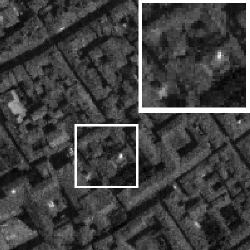}&
			\includegraphics[width=17mm, height = 20mm]{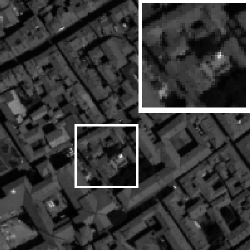}\\
			\includegraphics[width=17mm, height = 20mm]{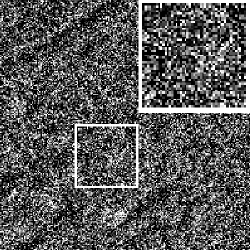}&
			\includegraphics[width=17mm, height = 20mm]{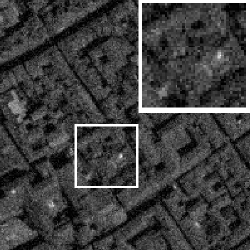}&
			\includegraphics[width=17mm, height = 20mm]{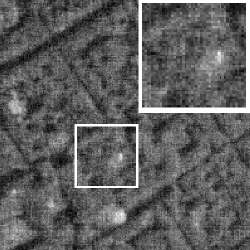}&
			\includegraphics[width=17mm, height = 20mm]{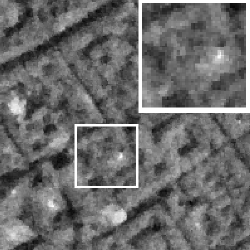}&
			\includegraphics[width=17mm, height = 20mm]{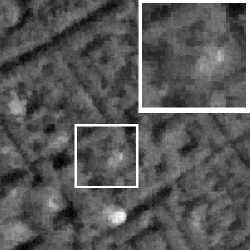}&
			\includegraphics[width=17mm, height = 20mm]{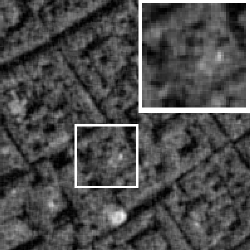}&
			\includegraphics[width=17mm, height = 20mm]{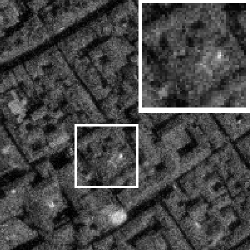}&
			\includegraphics[width=17mm, height = 20mm]{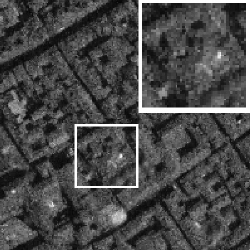}&
			\includegraphics[width=17mm, height = 20mm]{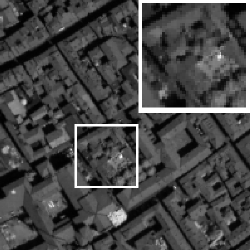}\\
			\includegraphics[width=17mm, height = 20mm]{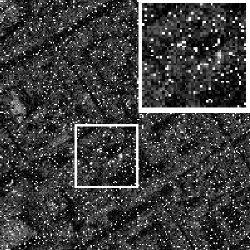}&
			\includegraphics[width=17mm, height = 20mm]{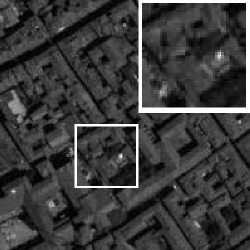}&
			\includegraphics[width=17mm, height = 20mm]{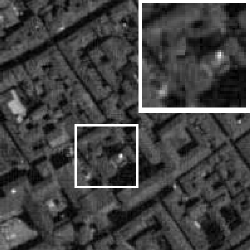}&
			\includegraphics[width=17mm, height = 20mm]{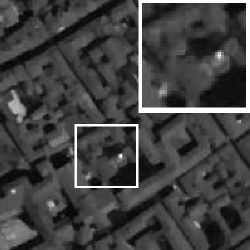}&
			\includegraphics[width=17mm, height = 20mm]{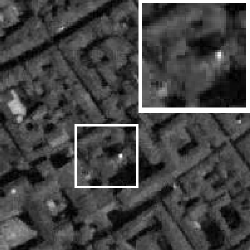}&
			\includegraphics[width=17mm, height = 20mm]{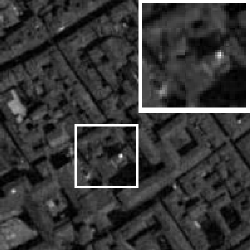}&
			\includegraphics[width=17mm, height = 20mm]{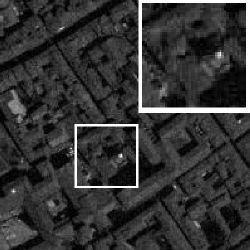}&
			\includegraphics[width=17mm, height = 20mm]{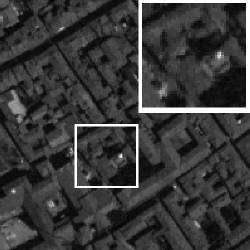}&
			\includegraphics[width=17mm, height = 20mm]{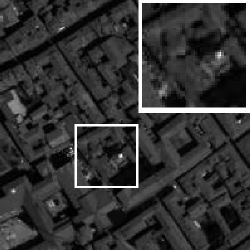}\\
			\includegraphics[width=17mm, height = 20mm]{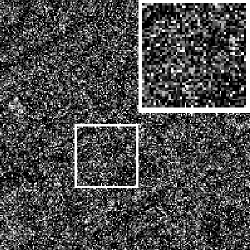}&
			\includegraphics[width=17mm, height = 20mm]{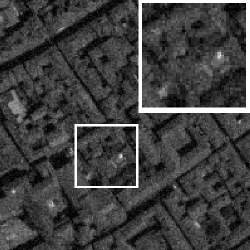}&
			\includegraphics[width=17mm, height = 20mm]{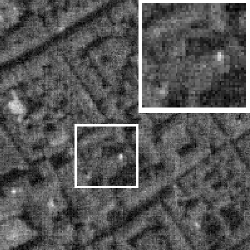}&
			\includegraphics[width=17mm, height = 20mm]{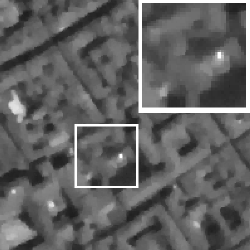}&
			\includegraphics[width=17mm, height = 20mm]{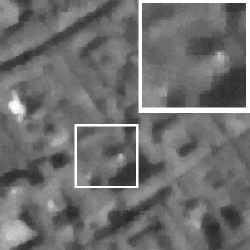}&
			\includegraphics[width=17mm, height = 20mm]{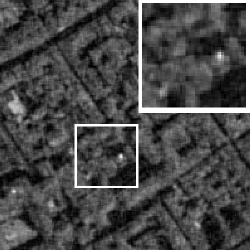}&
			\includegraphics[width=17mm, height = 20mm]{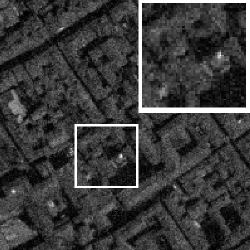}&
			\includegraphics[width=17mm, height = 20mm]{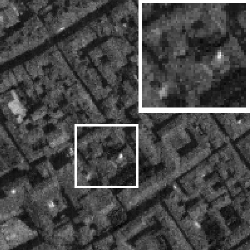}&
			\includegraphics[width=17mm, height = 20mm]{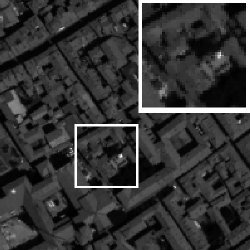}\\
			\scriptsize Observed & \scriptsize SNN & \scriptsize TNN & \scriptsize LRTV & \scriptsize TLR-SSTV & \scriptsize TCTV & \scriptsize SSCTV-RPCA & \scriptsize \textbf{Ours}  & \scriptsize Ground truth\\
		\end{tabular}
		\vspace{-0.2cm}
		\caption{Recovered Pavia HSI (bands 45, 66, 25, 35) from all competing methods under different noise conditions. From row 1 to row 4: Case 1 ($\text{G} = 0.1$), Case 2 ($\text{G} = 0.2$), Case 3 ($\text{S} = 0.1$, $\text{G} = 0.01$), Case 4 ($\text{S} = 0.1$, $\text{G} = 0.1$).}\label{fig.3}
		\vspace{-0.5cm}
	\end{figure*}
	
	\begin{figure*}[tp]
		\renewcommand{\arraystretch}{1.2}
		\setlength\tabcolsep{1pt}
		\centering
		\vspace{-0.1cm}
		\begin{tabular}{ccccccccc}
			\centering
			\includegraphics[width=17mm, height = 20mm]{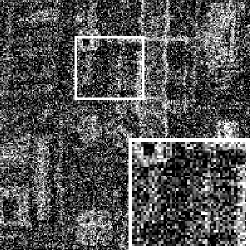}&
			\includegraphics[width=17mm, height = 20mm]{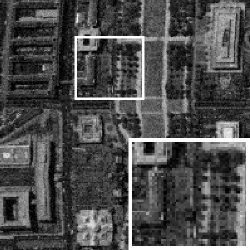}&
			\includegraphics[width=17mm, height = 20mm]{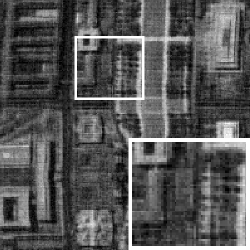}&
			\includegraphics[width=17mm, height = 20mm]{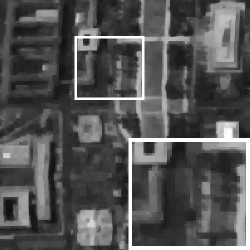}&
			\includegraphics[width=17mm, height = 20mm]{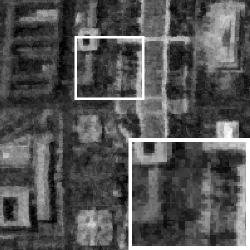}&
			\includegraphics[width=17mm, height = 20mm]{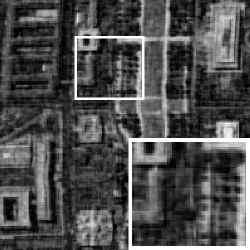}&
			\includegraphics[width=17mm, height = 20mm]{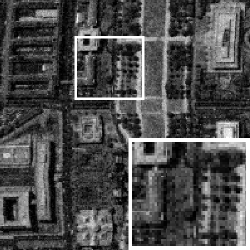}&
			\includegraphics[width=17mm, height = 20mm]{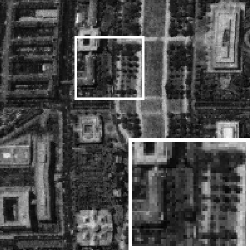}&
			\includegraphics[width=17mm, height = 20mm]{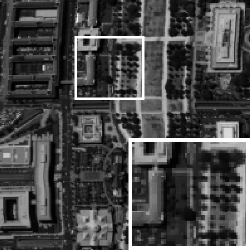}\\
			\includegraphics[width=17mm, height = 20mm]{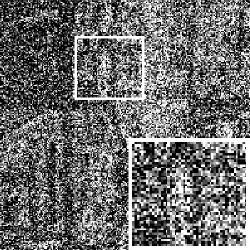}&
			\includegraphics[width=17mm, height = 20mm]{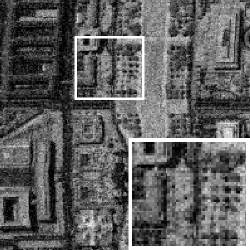}&
			\includegraphics[width=17mm, height = 20mm]{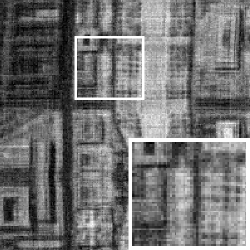}&
			\includegraphics[width=17mm, height = 20mm]{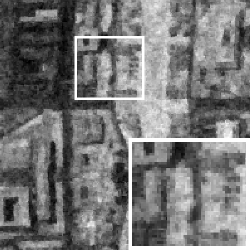}&
			\includegraphics[width=17mm, height = 20mm]{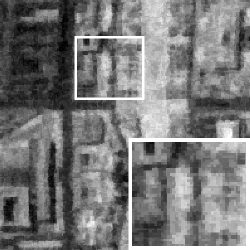}&
			\includegraphics[width=17mm, height = 20mm]{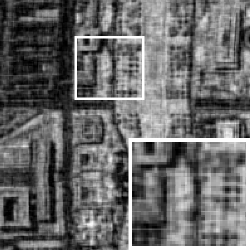}&
			\includegraphics[width=17mm, height = 20mm]{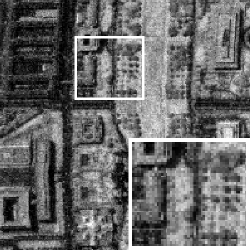}&
			\includegraphics[width=17mm, height = 20mm]{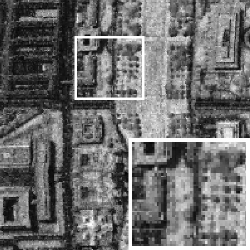}&
			\includegraphics[width=17mm, height = 20mm]{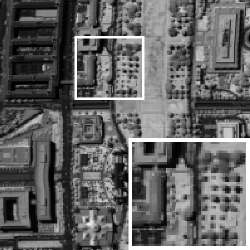}\\
			\includegraphics[width=17mm, height = 20mm]{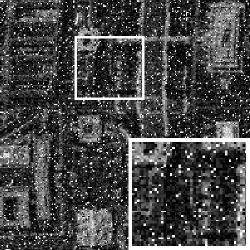}&
			\includegraphics[width=17mm, height = 20mm]{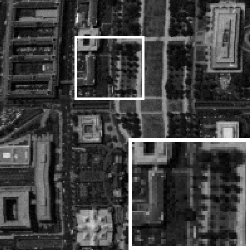}&
			\includegraphics[width=17mm, height = 20mm]{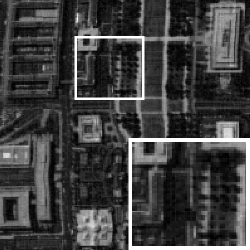}&
			\includegraphics[width=17mm, height = 20mm]{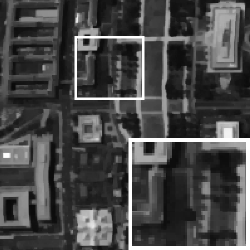}&
			\includegraphics[width=17mm, height = 20mm]{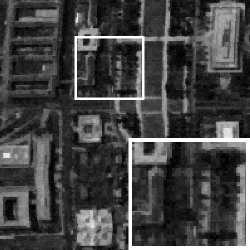}&
			\includegraphics[width=17mm, height = 20mm]{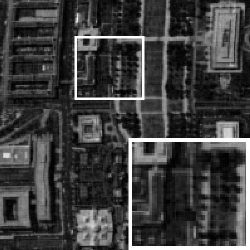}&
			\includegraphics[width=17mm, height = 20mm]{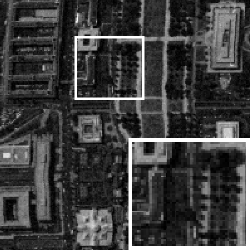}&
			\includegraphics[width=17mm, height = 20mm]{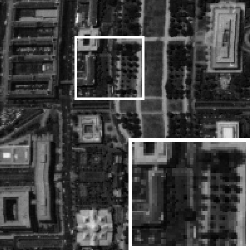}&
			\includegraphics[width=17mm, height = 20mm]{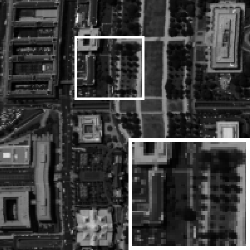}\\
			\includegraphics[width=17mm, height = 20mm]{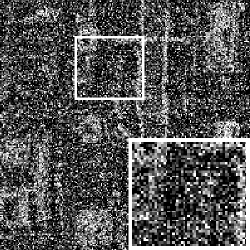}&
			\includegraphics[width=17mm, height = 20mm]{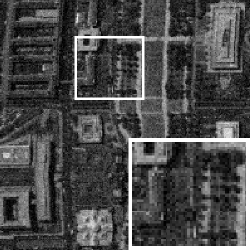}&
			\includegraphics[width=17mm, height = 20mm]{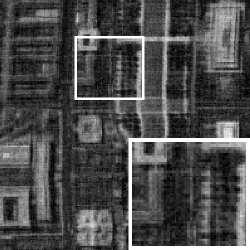}&
			\includegraphics[width=17mm, height = 20mm]{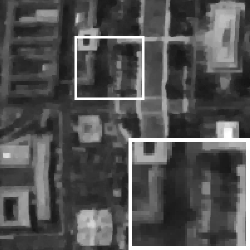}&
			\includegraphics[width=17mm, height = 20mm]{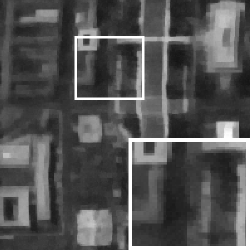}&
			\includegraphics[width=17mm, height = 20mm]{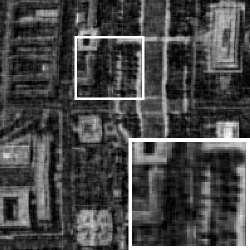}&
			\includegraphics[width=17mm, height = 20mm]{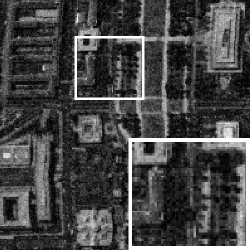}&
			\includegraphics[width=17mm, height = 20mm]{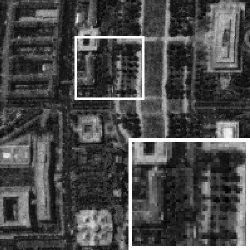}&
			\includegraphics[width=17mm, height = 20mm]{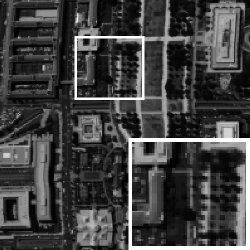}\\
			\scriptsize Observed & \scriptsize SNN & \scriptsize TNN & \scriptsize LRTV & \scriptsize TLR-SSTV & \scriptsize TCTV & \scriptsize SSCTV-RPCA & \scriptsize \textbf{Ours}  & \scriptsize Ground truth\\
		\end{tabular}
		\vspace{-0.2cm}
		\caption{Recovered DCMall HSI (bands 45, 66, 25, 35) from all competing methods under different noise conditions. From row 1 to row 4: Case 1 ($\text{G} = 0.1$), Case 2 ($\text{G} = 0.2$), Case 3 ($\text{S} = 0.1$, $\text{G} = 0.01$), Case 4 ($\text{S} = 0.1$, $\text{G} = 0.1$).}\label{fig.4}
		\vspace{-0.5cm}
	\end{figure*}
	
	\begin{figure}[tp]
		\renewcommand{\arraystretch}{0.5}
		\setlength\tabcolsep{0.5pt}
		\centering
		\vspace{-0.2cm}
		\begin{tabular}{ccc}
			\centering
			\tiny 14.892/0.195/449.319 & \tiny 29.344/0.906/86.689 & \tiny 28.426/0.821/102.510 \\
			\includegraphics[width=26mm, height = 21mm]{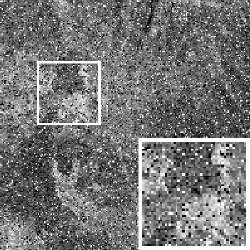}&
			\includegraphics[width=26mm, height = 21mm]{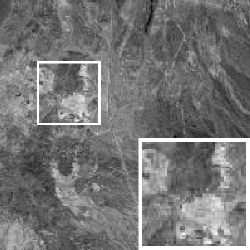}&
			\includegraphics[width=26mm, height = 21mm]{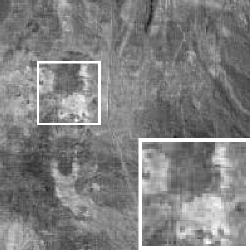}\\
			\scriptsize {Noisy} & \scriptsize SNN & \scriptsize TNN \\
			\tiny 28.248/0.715/97.467 & \tiny 26.911/0.681/120.428 & \tiny 30.292/0.864/83.351\\
			\includegraphics[width=26mm, height = 21mm]{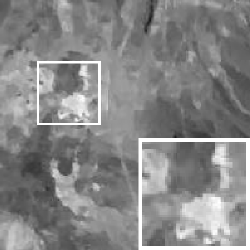}&
			\includegraphics[width=26mm, height = 21mm]{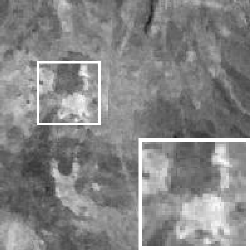}&
			\includegraphics[width=26mm, height = 21mm]{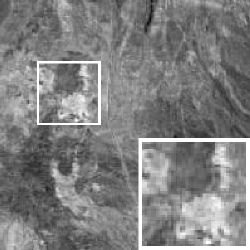}\\
			\scriptsize LRTV & \scriptsize TLR-SSTV & \scriptsize TCTV \\
			\tiny 29.831/0.846/84.529 & \tiny \textbf{32.403}/\textbf{0.920}/\textbf{67.176} & \tiny PSNR/SSIM/ERGAS\\
			\includegraphics[width=26mm, height = 21mm]{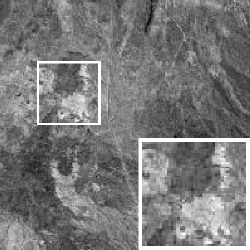}&
			\includegraphics[width=26mm, height = 21mm]{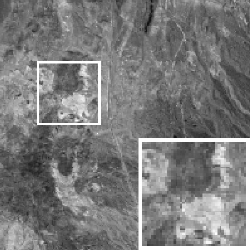}&
			\includegraphics[width=26mm, height = 21mm]{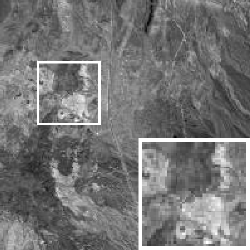}\\
			\scriptsize SSCTV-RPCA & \scriptsize \textbf{Ours} & \scriptsize {Ground truth}\\
		\end{tabular}
		\vspace{-0.2cm}
		\caption{Recovered RemoteImage HSI (band 25) from all competing methods under mixed noise of $\text{S}=0.1, \text{G}=0.01$.}\label{fig.5}
		\vspace{-0.3cm}
	\end{figure}
	
	\subsection{Experiments on Color Video Background Modeling}
	This subsection presents numerical experiments conducted on the publicly available SBMnet dataset$\footnote{http://www.SceneBackgroundModeling.net}$ \cite{jodoin2017extensive} to demonstrate the capability of our method in extracting background models from color video sequences. The performance of our proposed method is evaluated using four metrics recommended by the SBMnet dataset: AGE \cite{BOUWMANS20173}, PSNR, MSSSIM, and CQM \cite{Yalman2013}. For these metrics, higher values of PSNR, MSSSIM, and CQM indicate better background modeling, whereas a lower AGE corresponds to superior performance. Comparisons are made between our method and several state-of-the-art approaches, including ABM \cite{avola2017adaptive}, BE-AAPSA\cite{ramirez2017temporal}, DMD\cite{2014Dynamic}, Q-DMD\cite{HAN2022103560}, ETRPCA\cite{9170824} and $p$-TRPCA\cite{YAN2024104520}. The parameters of the proposed method are configured as follows: $\lambda_1 = 0.001$, $\lambda_2 = 0.001$ and $\lambda_3 = 0.1$.
	
	To evaluate the performance of the proposed method for color video background modeling, we conducted simulation experiments on six videos from the SBMnet dataset. Representative frames of these eight videos are illustrated in Fig. \ref{fig.6}. The spatial resolutions of these videos range from $208 \times 296$ to $480 \times 720$. The ``511'' video, originally at $480 \times 640$ resolution, was down-sampled by a factor of 2 to maintain manageable memory requirements on personal computers. For videos with frame counts substantially exceeding 200, we randomly selected 200 consecutive frames containing both foreground and background information, discarding temporally static frames to enable efficient processing.
	
	\begin{figure}[tp]
		\renewcommand{\arraystretch}{0.5}
		\setlength\tabcolsep{1.5pt}
		\centering
		\vspace{-0.2cm}
		\begin{tabular}{ccc}
			\centering
			\includegraphics[width=25.5mm,height=16mm]{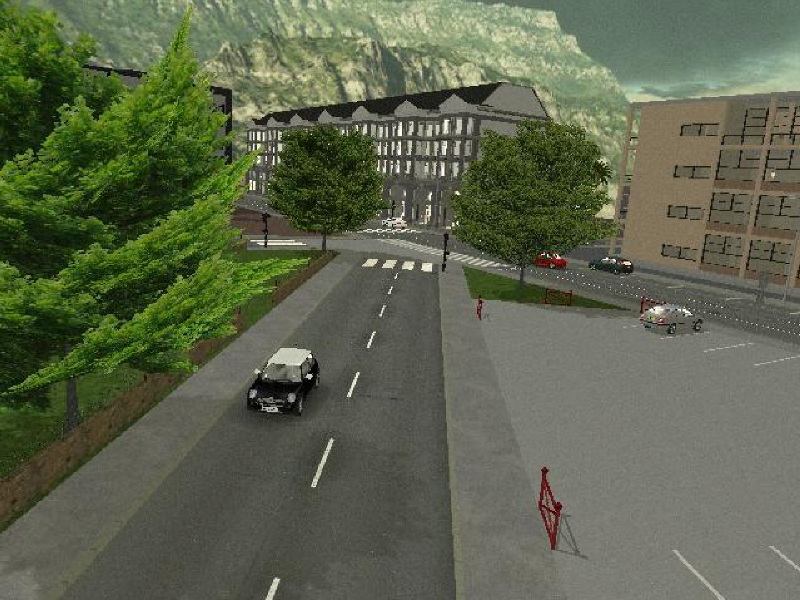}&
			\includegraphics[width=25.5mm,height=16mm]{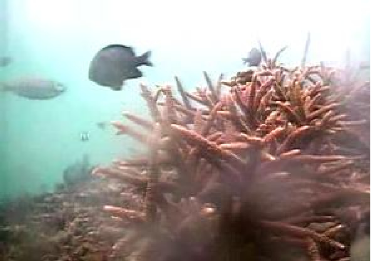}&
			\includegraphics[width=25.5mm,height=16mm]{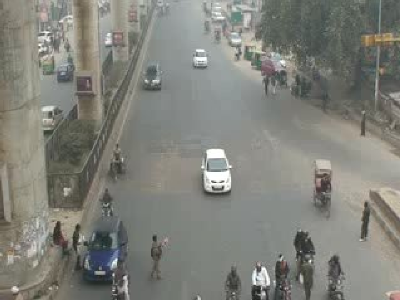}\\
			\footnotesize 511 & \footnotesize Blurred  & \footnotesize boulevardJam \\
			\includegraphics[width=25.5mm,height=16mm]{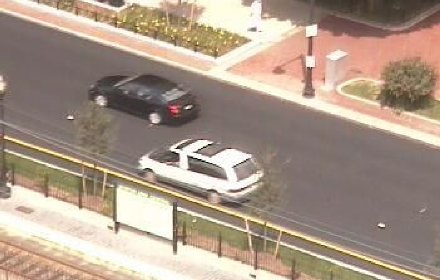}&
			\includegraphics[width=25.5mm,height=16mm]{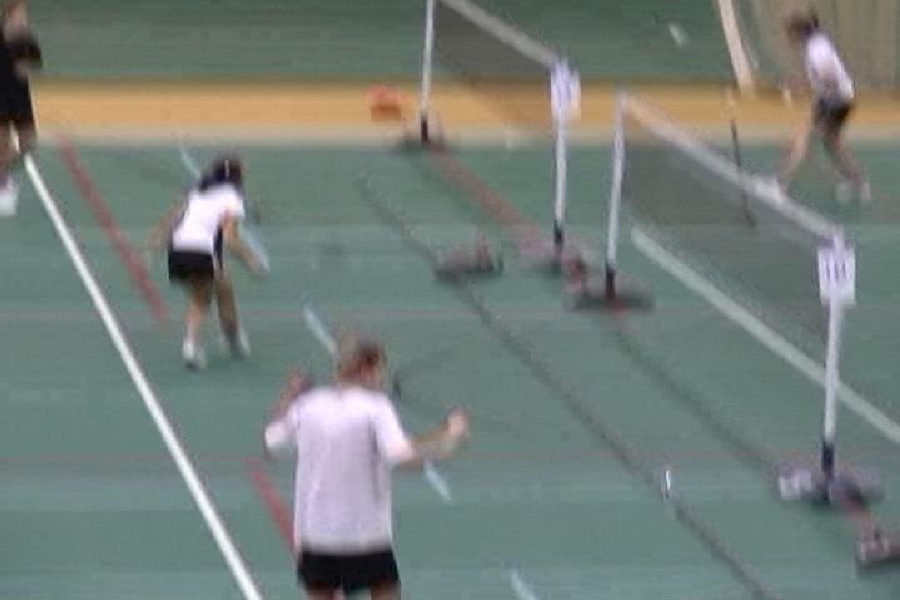}&
			\includegraphics[width=25.5mm,height=16mm]{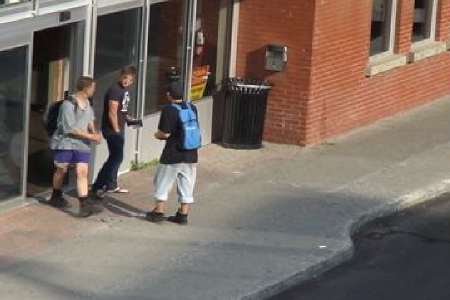}\\
			\footnotesize boulevard & \footnotesize Badminton  & \footnotesize BusStation\\
		\end{tabular}
		\vspace{-0.2cm}
		\caption{Representative frames from six videos of the SBMnet dataset.}\label{fig.6}
		\vspace{-0.3cm}
	\end{figure}
	
	The background generation results produced by various methods on the SBMnet dataset are presented in Fig. \ref{fig:SBMvisu}. It can be observed that our proposed method yields the most visually favorable background images compared to other approaches. For the sequences ``511'', ``Blurred'', and ``boulevard'', which feature long background exposure times and mild object motion, most methods produce clear backgrounds with satisfactory perceptual quality. Notably, our method achieves the best results among them, with the only exceptions being BE-AAPSA and 
	$p$-TRPCA, which underperform on the ``Blurred'' sequence. In contrast, for videos where the background is visible only intermittently (e.g., ``boulevard-Jam''), all compared methods, including ours, exhibit shadow effects from moving objects in the generated backgrounds. Nevertheless, our method achieves relatively better results under these challenging conditions.
	
	The ``badminton'' sequence involves camera jitter, where instability (e.g., vibration) coupled with foreground motion interferes with background generation. While most competing methods suffer from blurring effects, our method remains robust and delivers superior performance. ``BusStation'' belongs to the intermittent motion category, where the foreground moves, pauses briefly, and then resumes movement. This behavior often leads to ghosting artifacts in the reconstructed background. On this sequence, our method clearly outperforms ABM, DMD, Q-DMD, ETRPCA, and $p$-TRPCA.
	
	Quantitative evaluations are summarized in Table \ref{sbm}. As shown in Table \ref{sbm}, our method consistently achieves the best performance among all compared methods across all metrics. Specifically, our method achieves the lowest AGE values across all videos, with particularly notable improvements on ``Basic (511)'' and ``Jitter (Badminton)''. With regard to CQM, the most comprehensive metric considered in this study, our method ranks first on all six videos, surpassing the second-best method by an average margin of approximately $3.1\%$. Moreover, our method attains the highest MSSSIM and PSNR scores on every video as well. Together, these quantitative results validate the effectiveness and robustness of our proposed approach for color video background modeling.

	\begin{figure*}[tp]
		\renewcommand{\arraystretch}{1.2}
		\setlength\tabcolsep{1pt}
		\centering
		\vspace{-0.1cm}
		\begin{tabular}{ccccccccc}
			\centering
			\scriptsize{Original} & \scriptsize{ABM} & \scriptsize{BE-AA-PSA} & \scriptsize{DMD} & \scriptsize{Q-DMD}& \scriptsize{ETRPCA} & \scriptsize{$p$-TRPCA} &\scriptsize{\textbf{Ours}} & \scriptsize{Ground truth} \\
			\includegraphics[width=1.75cm,height=1.6cm]{SBM/ori/511_ori-eps-converted-to.pdf} &
			\includegraphics[width=1.75cm,height=1.6cm]{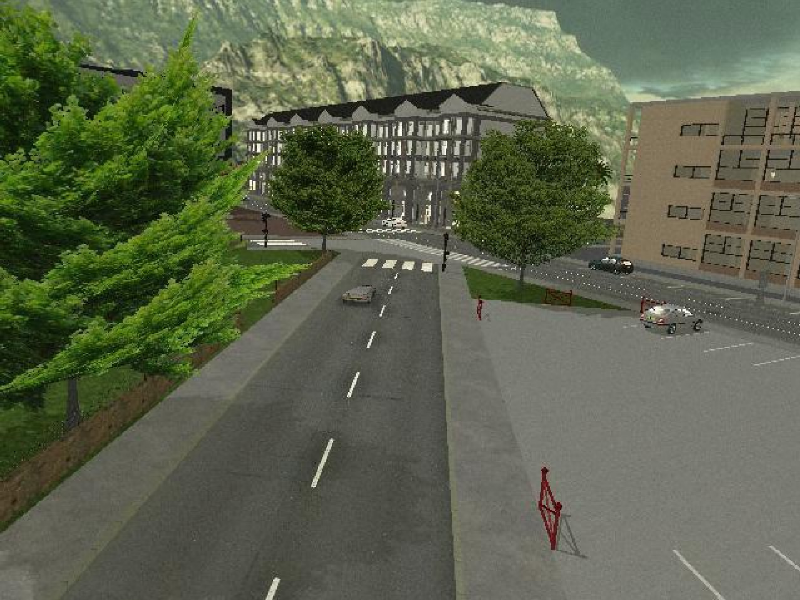} &
			\includegraphics[width=1.75cm,height=1.6cm]{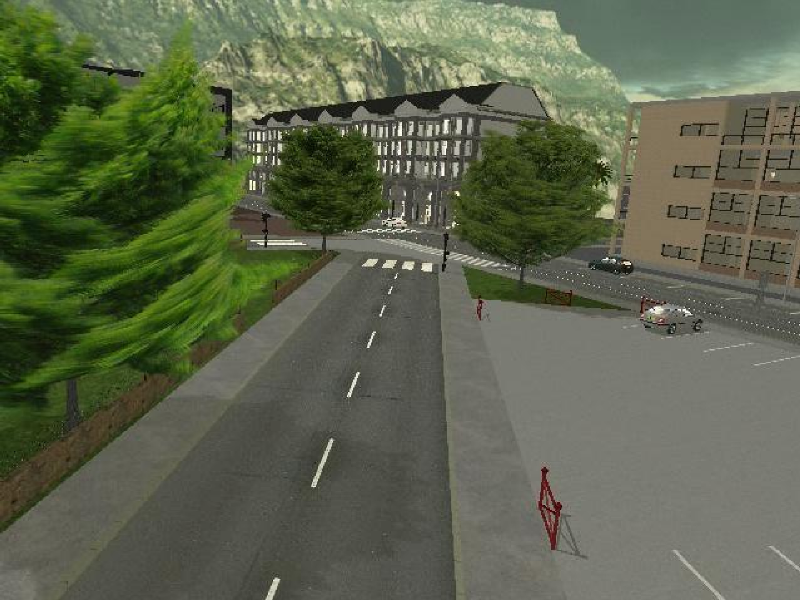} &
			\includegraphics[width=1.75cm,height=1.6cm]{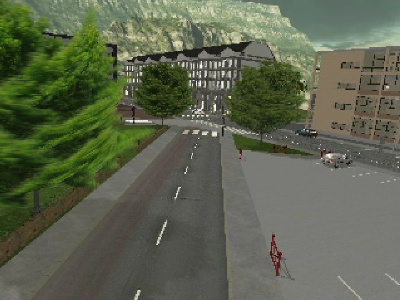} &
			\includegraphics[width=1.75cm,height=1.6cm]{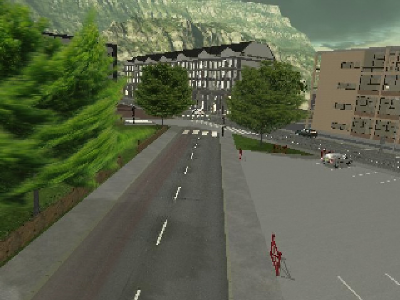} &
			\includegraphics[width=1.75cm,height=1.6cm]{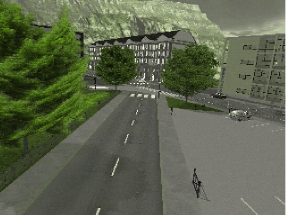} &
			\includegraphics[width=1.75cm,height=1.6cm]{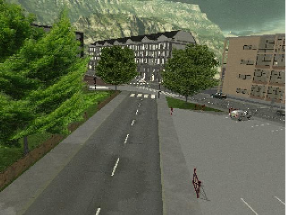} &
			\includegraphics[width=1.75cm,height=1.6cm]{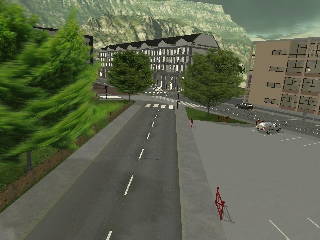} &
			\includegraphics[width=1.75cm,height=1.6cm]{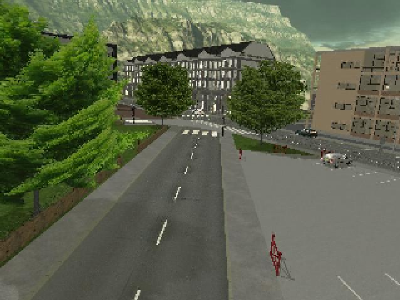} \\
			\hline
			\includegraphics[width=1.75cm,height=1.6cm]{SBM/ori/Blurred_ori-eps-converted-to.pdf} &
			\includegraphics[width=1.75cm,height=1.6cm]{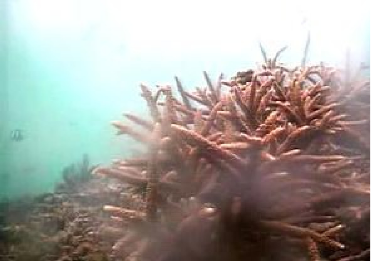} &
			\includegraphics[width=1.75cm,height=1.6cm]{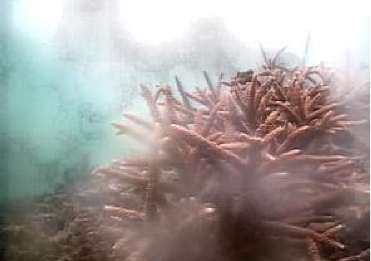} &
			\includegraphics[width=1.75cm,height=1.6cm]{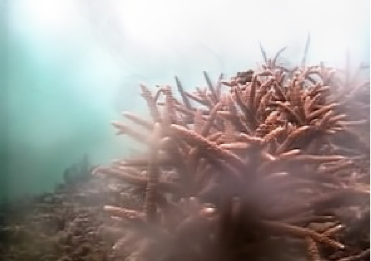} &
			\includegraphics[width=1.75cm,height=1.6cm]{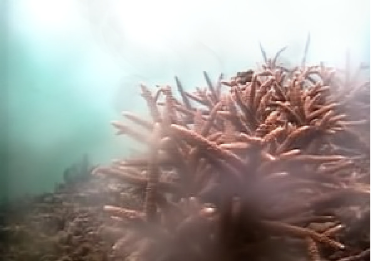}&
			\includegraphics[width=1.75cm,height=1.6cm]{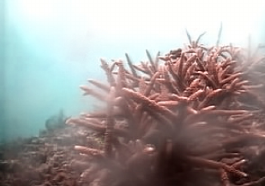} &
			\includegraphics[width=1.75cm,height=1.6cm]{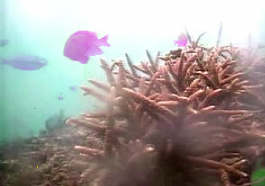} &
			\includegraphics[width=1.75cm,height=1.6cm]{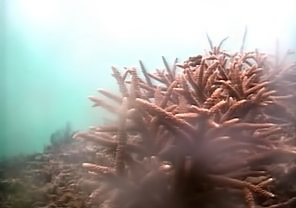}& \includegraphics[width=1.75cm,height=1.6cm]{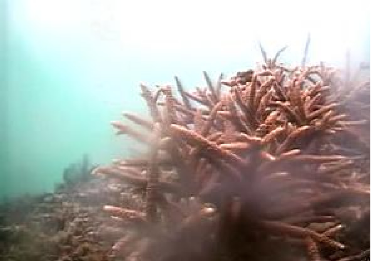} \\
			\hline
			\includegraphics[width=1.75cm,height=1.6cm]{SBM/ori/boulvardJam_ori-eps-converted-to.pdf} &
			\includegraphics[width=1.75cm,height=1.6cm]{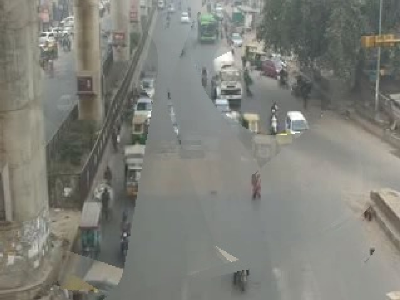} &
			\includegraphics[width=1.75cm,height=1.6cm]{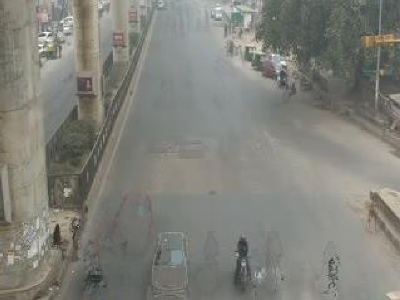} &
			\includegraphics[width=1.75cm,height=1.6cm]{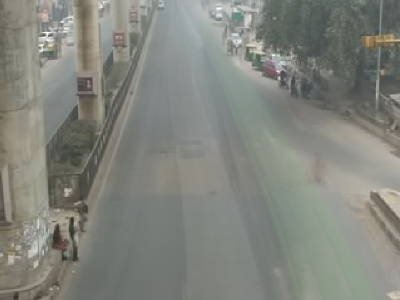} &
			\includegraphics[width=1.75cm,height=1.6cm]{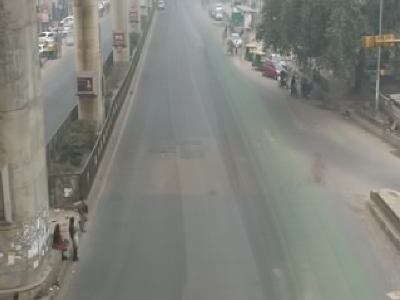}&
			\includegraphics[width=1.75cm,height=1.6cm]{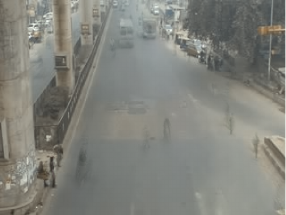} &
			\includegraphics[width=1.75cm,height=1.6cm]{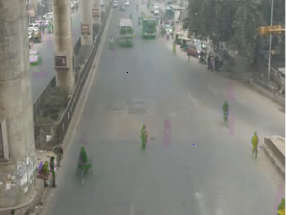} &
			\includegraphics[width=1.75cm,height=1.6cm]{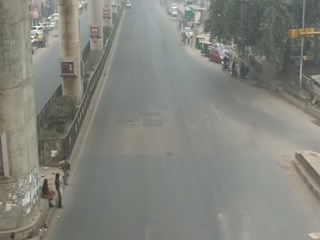} &
			\includegraphics[width=1.75cm,height=1.6cm]{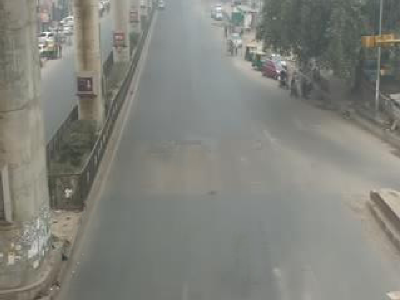}  \\
			\hline
			\includegraphics[width=1.75cm,height=1.6cm]{SBM/ori/boulevard_ori-eps-converted-to.pdf} &
			\includegraphics[width=1.75cm,height=1.6cm]{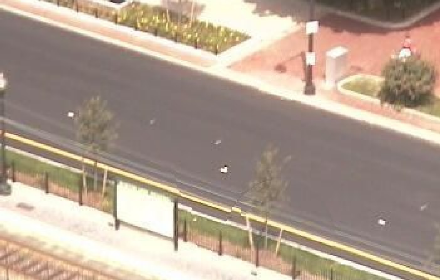} &
			\includegraphics[width=1.75cm,height=1.6cm]{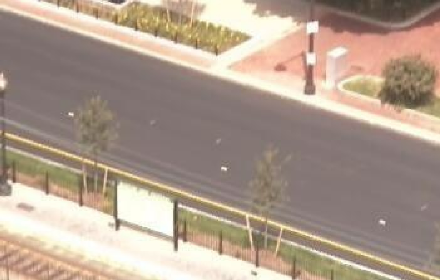} &
			\includegraphics[width=1.75cm,height=1.6cm]{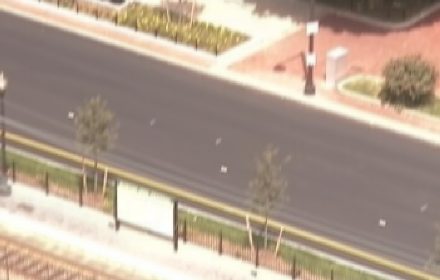} &
			\includegraphics[width=1.75cm,height=1.6cm]{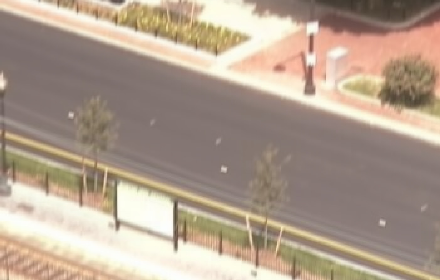}&	\includegraphics[width=1.75cm,height=1.6cm]{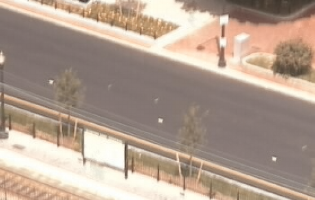} &
			\includegraphics[width=1.75cm,height=1.6cm]{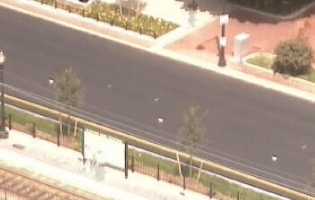} &
			\includegraphics[width=1.75cm,height=1.6cm]{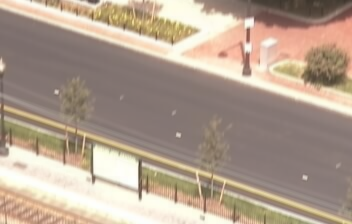} &	\includegraphics[width=1.75cm,height=1.6cm]{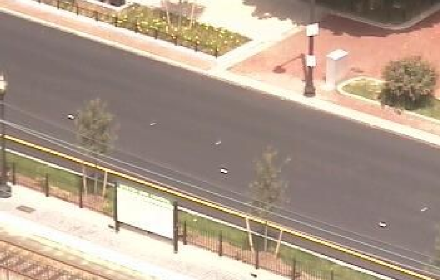} \\
			\hline
			\includegraphics[width=1.75cm,height=1.6cm]{SBM/ori/badminton_ori-eps-converted-to.pdf} &
			\includegraphics[width=1.75cm,height=1.6cm]{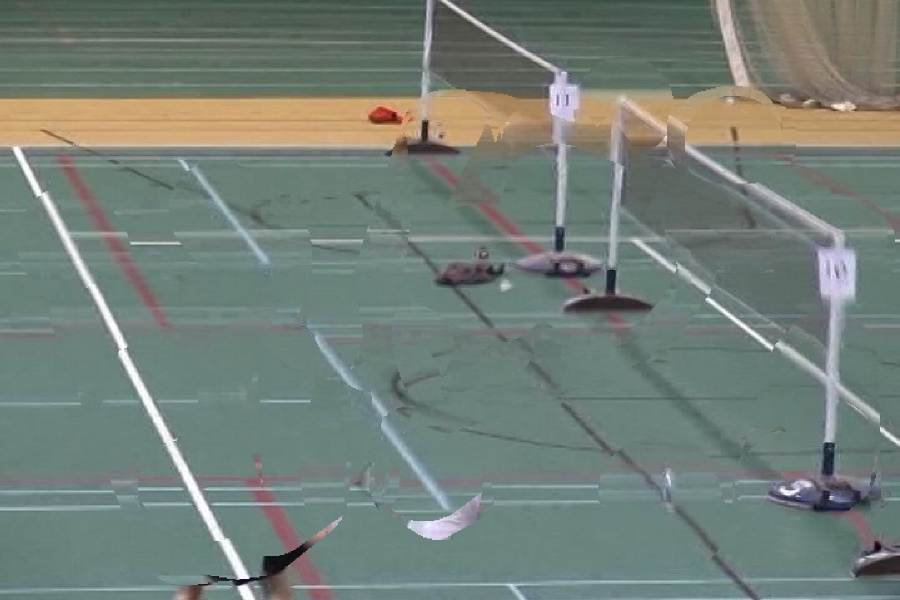} &
			\includegraphics[width=1.75cm,height=1.6cm]{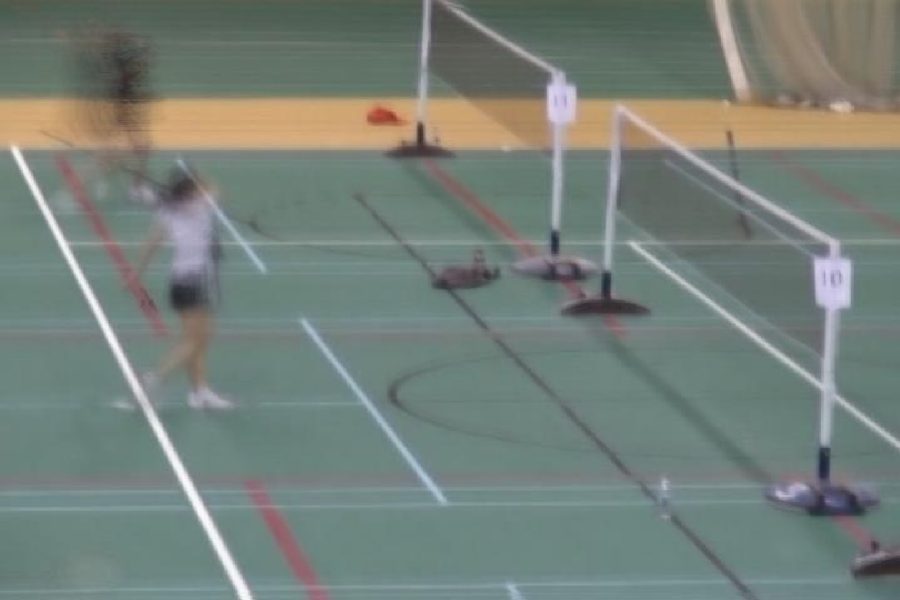} &
			\includegraphics[width=1.75cm,height=1.6cm]{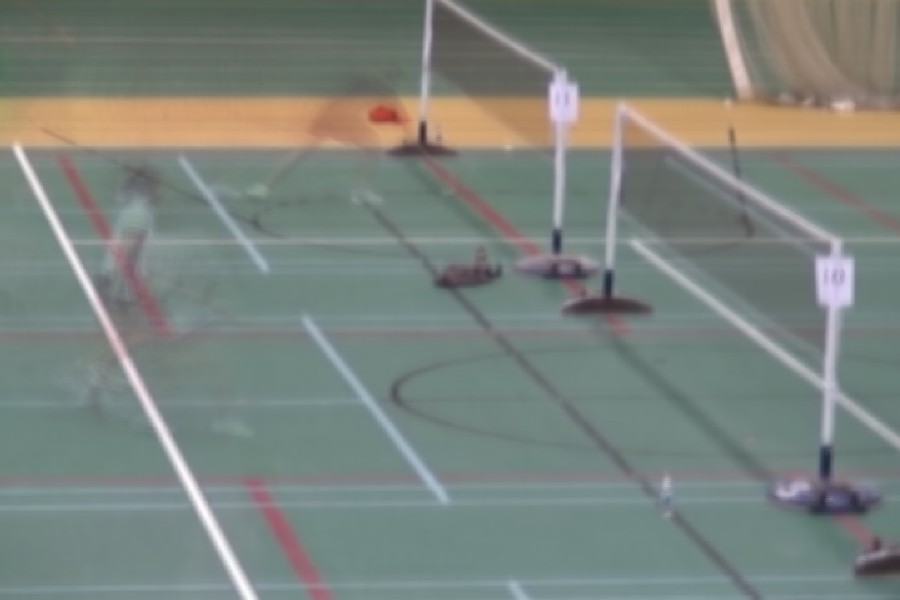} &
			\includegraphics[width=1.75cm,height=1.6cm]{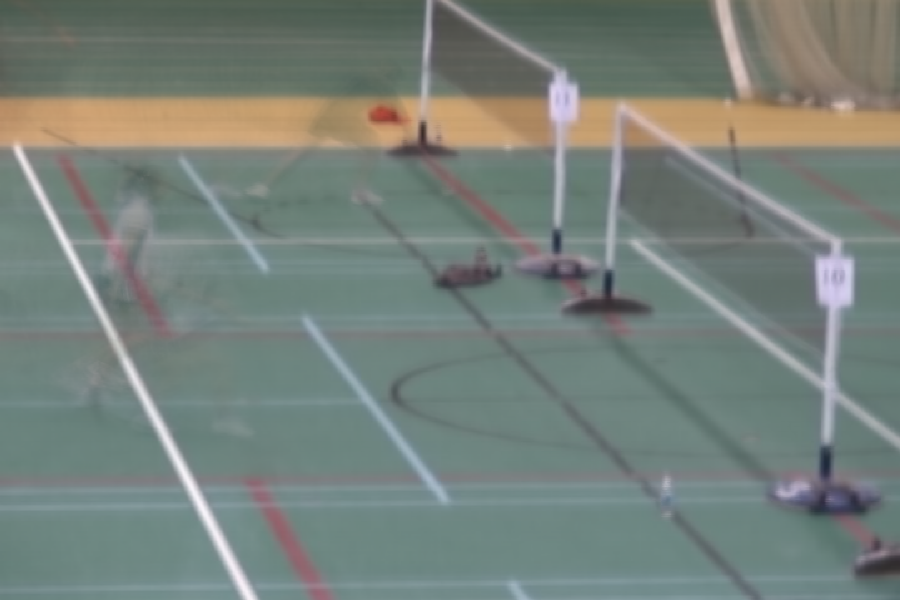} &
			\includegraphics[width=1.75cm,height=1.6cm]{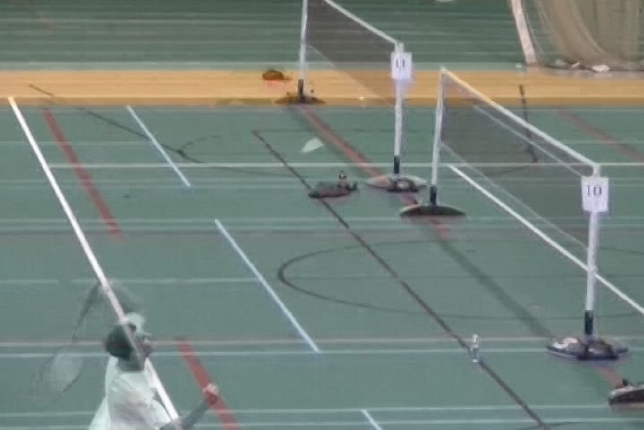} &
			\includegraphics[width=1.75cm,height=1.6cm]{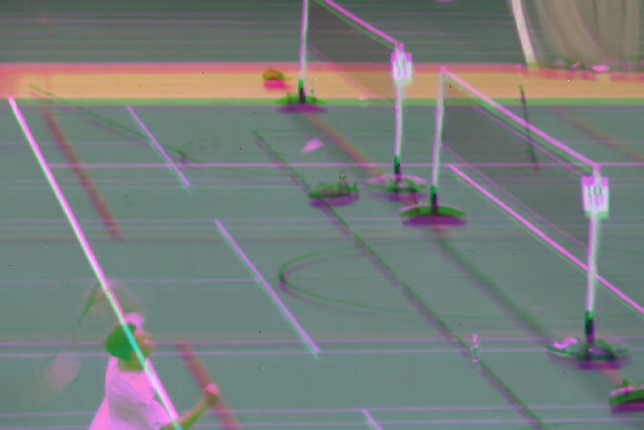} &
			\includegraphics[width=1.75cm,height=1.6cm]{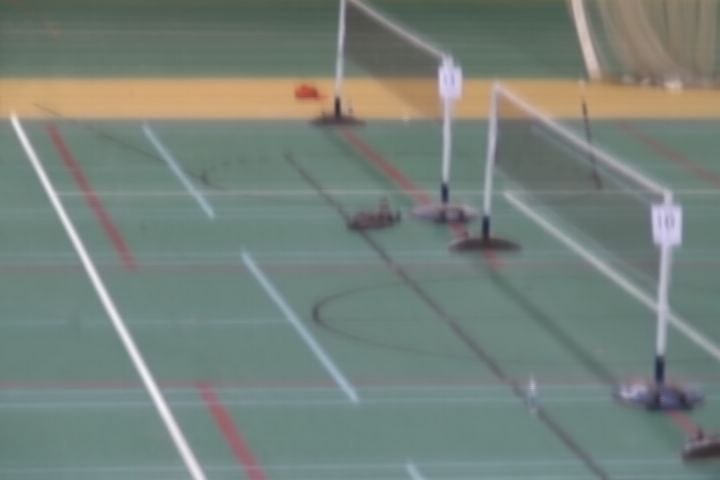}&	\includegraphics[width=1.75cm,height=1.6cm]{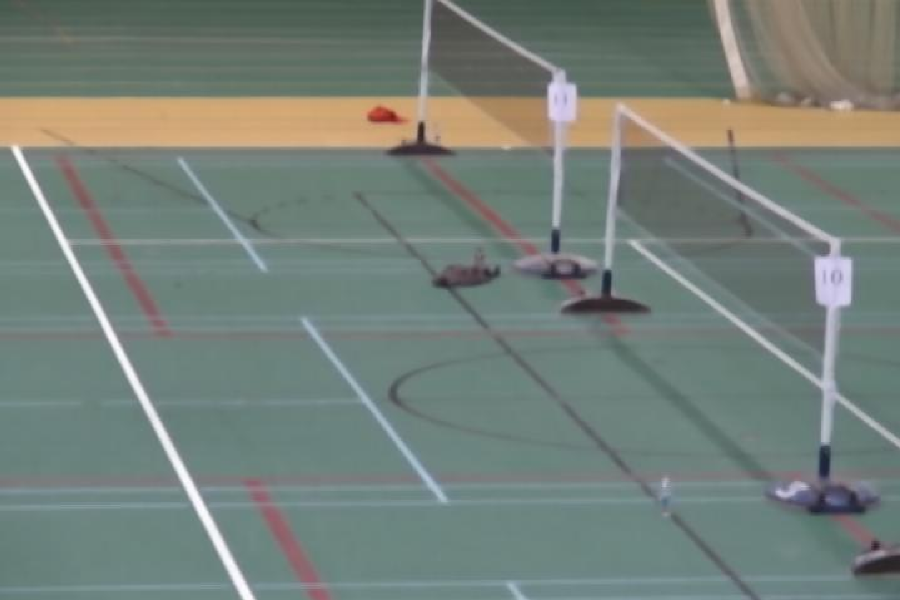} \\
			\hline
			\includegraphics[width=1.75cm,height=1.6cm]{SBM/ori/busStation_ori-eps-converted-to.pdf} &
			\includegraphics[width=1.75cm,height=1.6cm]{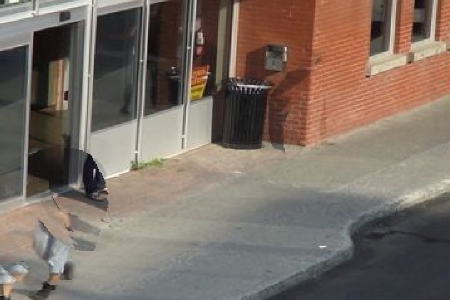} &
			\includegraphics[width=1.75cm,height=1.6cm]{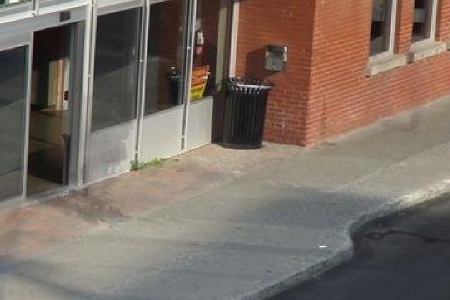} &
			\includegraphics[width=1.75cm,height=1.6cm]{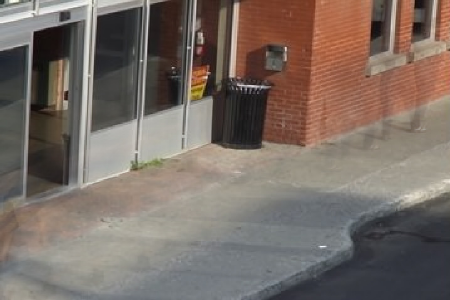} &
			\includegraphics[width=1.75cm,height=1.6cm]{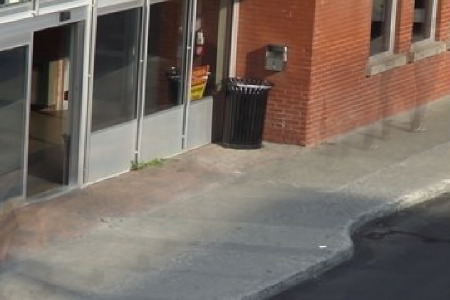}& 	\includegraphics[width=1.75cm,height=1.6cm]{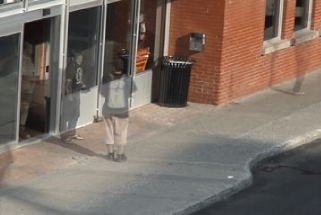} &
			\includegraphics[width=1.75cm,height=1.6cm]{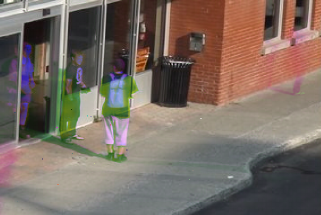} &
			\includegraphics[width=1.75cm,height=1.6cm]{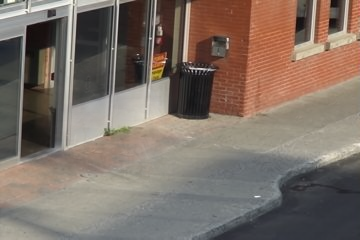}&
			\includegraphics[width=1.75cm,height=1.6cm]{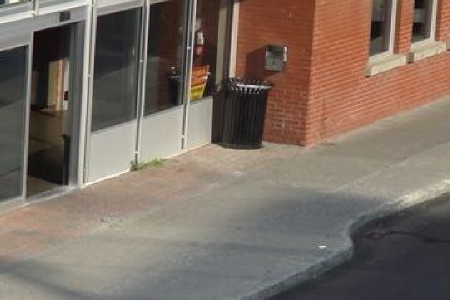} \\
			\multicolumn{1}{c}{{(a)}} & \multicolumn{1}{c}{{(b)}} & \multicolumn{1}{c}{{(c)}} & \multicolumn{1}{c}{{(d)}} & \multicolumn{1}{c}{{(e)}} & \multicolumn{1}{c}{{(f)}} & \multicolumn{1}{c}{{(g)}} & \multicolumn{1}{c}{{(h)}} & \multicolumn{1}{c}{{(i)}} \\
		\end{tabular}
		\vspace{-0.2cm}
		\caption{\label{fig:SBMvisu} Representative frames of six videos from the SBMnet dataset. (a) Original frames; (b)–(h) respectively show the background images generated by ABM, BE‑AAPSA, DMD, Q‑DMD, ETRPCA, $p$-TRPCA, and our proposed method; (i) Ground truth. The videos are arranged from top to bottom as follows: ``511'', ``Blurred'', ``boulevardJam'', ``boulevard'', ``Badminton'', and ``BusStation''.}
		\vspace{-0.5cm}
	\end{figure*}

	\begin{table*}[htbp]
		\footnotesize
		\caption{Comparison of AGE, MSSSIM, PSNR, and CQM between proposed method and other state-of-the-art background modeling methods on SBMnet dataset. The bold numbers represent the best results for each metric on each video.}
		\centering
		\renewcommand\arraystretch{0.6}{
			\setlength{\tabcolsep}{1.8mm}{
				\begin{tabular}{lllllllllllll}
					\toprule
					Videos &  & ABM\cite{avola2017adaptive}   & \makecell[c]{BE-AA- \\ PSA}\cite{ramirez2017temporal}   & DMD\cite{2014Dynamic} & Q-DMD\cite{HAN2022103560}  & ETRPCA\cite{9170824} & $p$-TRPCA\cite{YAN2024104520} & \textbf{Ours} \\
					\midrule				
					\multirow{4}{*}{Basic (511)} & AGE$\downarrow$ & 6.1224 & 4.0511  & 4.0569 & 4.0425 & 6.0305 & 7.1467 & \textbf{3.7998}\\
					&  MSSSIM$\uparrow$ & 0.9407 & 0.9744 & 0.9780 & 0.9783 & 0.9601 & 0.9415 & \textbf{0.9830}
					\\ 
					&  PSNR$\uparrow$ & 26.5925 & 30.0319 & 30.2008 & 30.2637 & 27.8177 & 25.8395 & \textbf{30.7929} \\ 
					&  CQM$\uparrow$ & 28.5767 & 31.8292 & 31.8512 & 32.1079 & 29.5069 & 27.8292 & \textbf{32.8066} \\ 				
					\midrule 
					\multirow{4}{*}{Basic (Blurred)} & AGE$\downarrow$ & 2.4574  & 15.2057  & 11.3377 & 9.4219 & 2.9908 & 5.1130 & \textbf{2.2926}\\
					&  MSSSIM$\uparrow$ & 0.9906  & 0.8924  & 0.9561 & 0.9596 & 0.9927 & 0.9122 & \textbf{0.9939}
					\\ 
					&  PSNR$\uparrow$ & 36.5219  &  22.4556 & 24.3609 & 25.5840 & 35.6016 & 25.9039 & \textbf{37.6021} \\ 				
					&  CQM$\uparrow$ & 36.8815  &  23.3364 & 25.3249 & 26.4187 & 36.6346 & 26.5555 & \textbf{37.9613} \\ 
					\midrule
					\multirow{4}{*}{\makecell[l]{Clutter \\  (boulevardJam)}} & AGE$\downarrow$ & 7.6215 & 5.1418 & 5.6188 & 3.6353 &3.4880 & 4.3483 & \textbf{2.4745} \\
					&  MSSSIM$\uparrow$ & 0.7546 & 0.9219 & 0.9541 & 0.9325 &0.9364 & 0.8965 & \textbf{0.9703}
					\\ 
					&  PSNR$\uparrow$ & 24.7166 & 28.9114 & 30.2312 & 31.5366 & 31.2035 & 28.8058 & \textbf{32.9818} \\ 
					&  CQM$\uparrow$ & 26.0921 & 30.0986 & 31.0495 & 32.6002 & 32.2364 & 29.2690 & \textbf{34.0861}  \\ 
					\midrule
					\multirow{4}{*}{\makecell[l]{Jitter \\ (boulevard)} } & AGE$\downarrow$ & 11.4377 & 10.8262 & 9.4585 & 9.4098 & 13.1872  & 16.0369 & \textbf{9.1482} \\
					&  MSSSIM$\uparrow$ & 0.8776 & 0.8821 & 0.9116 & 0.9116 &  0.8193 &0.7464 & \textbf{0.9138}
					\\ 
					&  PSNR$\uparrow$ & 20.6235 & 21.1393 & 22.6940 & 22.7186 & 20.2411 & 18.7159 & \textbf{22.8243} \\ 
					&  CQM$\uparrow$ & 22.0462 & 22.5861 & 24.1742 & 24.2013 & 21.7597 &20.2132 & \textbf{24.3108}  \\ 
					\midrule
					\multirow{4}{*}{\makecell[l]{Jitter \\ (Badminton)}  } & AGE$\downarrow$ & 7.7833  &  4.3975 & 5.6188 & 3.9600 & 11.3827 & 6.9351 & \textbf{3.5479} \\
					&  MSSSIM$\uparrow$ & 0.7872 & 0.9204  & 0.9541 & 0.9570 & 0.6104 & 0.7993 & \textbf{0.9586}	\\ 
					&  PSNR$\uparrow$ & 23.9886 & 29.1352  & 30.2312 & 31.3778 & 21.6303 & 25.4854 & \textbf{31.8946} \\ 
					&  CQM$\uparrow$ & 24.9043 & 29.9490 & 31.0495 & 32.2622 & 22.8356 & 24.7705 & \textbf{32.7904}  \\ 
					\midrule
					\multirow{4}{*}{\makecell[l]{IntermittentMotion\\ (BusStation)} } & AGE$\downarrow$ & 6.5147 & 4.5206 & 6.4834 & 6.3356 & 7.1659 & 7.5664 & \textbf{4.7462} \\
					&  MSSSIM$\uparrow$ & 0.9128 & 0.9621 & 0.9512 & 0.9505 & 0.8851 & 0.8781 & \textbf{0.9808}	
					\\ 
					&  PSNR$\uparrow$ & 24.4467 & 30.0286 & 28.3237 & 28.0328 & 24.5661 & 24.0556 & \textbf{32.2511} \\ 
					&  CQM$\uparrow$ & 25.5588 & 30.9833 & 29.2942 & 29.1266 & 25.7711 & 23.6989 & \textbf{33.0502}  \\ 
					\bottomrule
					\label{sbm}
		\end{tabular}}}
	\end{table*}

	To further validate that the performance gain of our method stems from the proposed TOSS strategy rather than merely from the incorporation of TV regularization, we additionally compare our method with several TV-based low-rank tensor recovery approaches originally designed for hyperspectral image denoising, including LRTV \cite{he2015total}, TLR-SSTV \cite{chen2018tensor}, TCTV \cite{WangHailin2023}, and SSCTV-RPCA \cite{GAO2025128885}. Although these methods were developed for denoising tasks, they all follow a low-rank-plus-sparse (L+S) decomposition paradigm within the tensor RPCA framework, which is conceptually related to our background modeling formulation. Therefore, we adapt them to the video background extraction task by treating the low-rank component as the background model.
	
	As shown in the visual results (Fig. \ref{fig:SBMTV}) and quantitative comparisons (Table \ref{sbmtv}), our method consistently outperforms these TV-based baselines across all evaluation metrics. Specifically, for the ``Basic (511)" sequence, our method improves PSNR by approximately 3.4 dB over the best TV-based competitor (SSCTV-RPCA) and raises MSSSIM from 0.9515 to 0.9830. On the challenging ``Jitter (Badminton)" sequence, which exhibits severe camera motion, our method outperforms TCTV and SSCTV-RPCA by margins of 11.0 dB and 10.8 dB in PSNR, respectively, while also achieving the lowest AGE. For the ``IntermittentMotion (BusStation)" sequence, our method again surpasses all baselines, yielding a PSNR of 32.2511 dB compared to 22.4911 dB from TLR-SSTV and 19.5868 dB from SSCTV-RPCA. These quantitative results, together with the visual comparisons in Fig.~\ref{fig:SBMvisu}, demonstrate that the performance gain of our method stems from the proposed TOSS strategy rather than merely from the incorporation of TV regularization. By explicitly splitting the tensor into a dominant component lying in a prescribed subspace and a residual component in its orthogonal complement, and then applying TV regularization separately to each component, our TOSS framework achieves more accurate background separation with superior robustness across diverse video scenarios. This result is consistent with the theoretical properties established in the proposed framework.

	\begin{figure*}[tp]
		\renewcommand{\arraystretch}{1.2}
		\setlength\tabcolsep{1pt}
		\centering
		\vspace{-0.1cm}
		\begin{tabular}{ccccccc}
			\centering
			\scriptsize{Original} & \scriptsize{LRTV} & \scriptsize{TLR-SSTV} & \scriptsize{TCTV} & \scriptsize{SSCTV-RPCA} &\scriptsize{\textbf{Ours}} & \scriptsize{Ground truth} \\
			\includegraphics[width=2cm,height=1.6cm]{SBM/ori/511_ori-eps-converted-to.pdf} &
			\includegraphics[width=2cm,height=1.6cm]{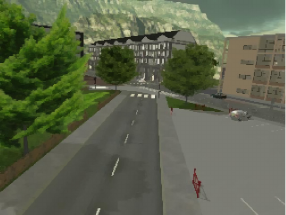} &
			\includegraphics[width=2cm,height=1.6cm]{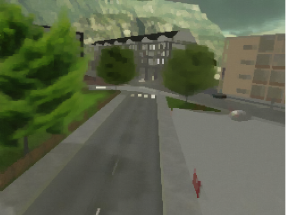} &
			\includegraphics[width=2cm,height=1.6cm]{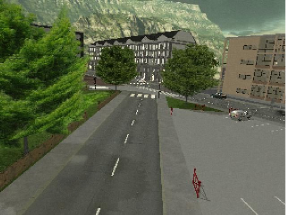} &
			\includegraphics[width=2cm,height=1.6cm]{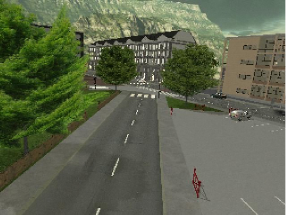} &
			\includegraphics[width=2cm,height=1.6cm]{SBM/Our_511.png} &
			\includegraphics[width=2cm,height=1.6cm]{SBM/GT_511-eps-converted-to.pdf} \\
			\hline
			\includegraphics[width=2cm,height=1.6cm]{SBM/ori/Blurred_ori-eps-converted-to.pdf} &
			\includegraphics[width=2cm,height=1.6cm]{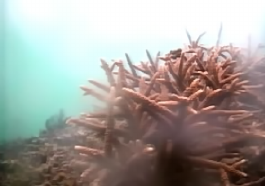} &
			\includegraphics[width=2cm,height=1.6cm]{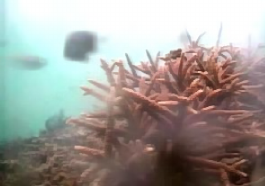} &
			\includegraphics[width=2cm,height=1.6cm]{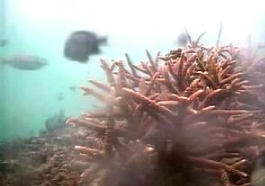} &
			\includegraphics[width=2cm,height=1.6cm]{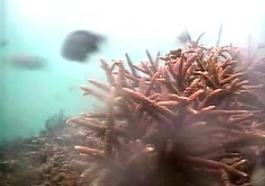}&
			\includegraphics[width=2cm,height=1.6cm]{SBM/Our_Blurred.png}& \includegraphics[width=2cm,height=1.6cm]{SBM/GT_Blurred-eps-converted-to.pdf} \\
			\hline
			\includegraphics[width=2cm,height=1.6cm]{SBM/ori/boulvardJam_ori-eps-converted-to.pdf} &
			\includegraphics[width=2cm,height=1.6cm]{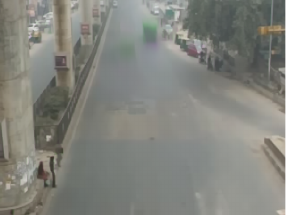} &
			\includegraphics[width=2cm,height=1.6cm]{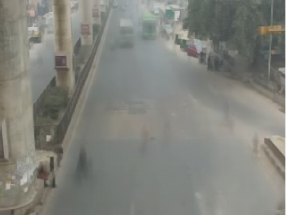} &
			\includegraphics[width=2cm,height=1.6cm]{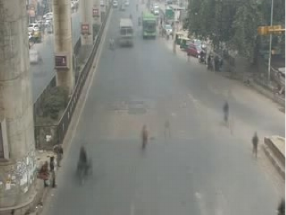} &
			\includegraphics[width=2cm,height=1.6cm]{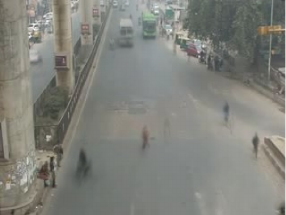}&
			\includegraphics[width=2cm,height=1.6cm]{SBM/Our_boulvardJam.png} &
			\includegraphics[width=2cm,height=1.6cm]{SBM/GT_boulvardJam-eps-converted-to.pdf}  \\
			\hline
			\includegraphics[width=2cm,height=1.6cm]{SBM/ori/boulevard_ori-eps-converted-to.pdf} &
			\includegraphics[width=2cm,height=1.6cm]{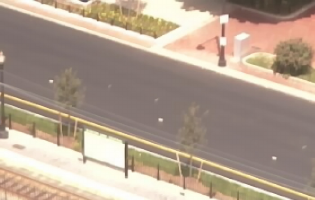} &
			\includegraphics[width=2cm,height=1.6cm]{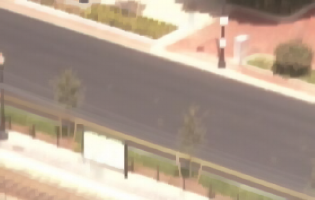} &
			\includegraphics[width=2cm,height=1.6cm]{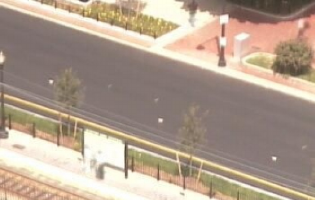} &
			\includegraphics[width=2cm,height=1.6cm]{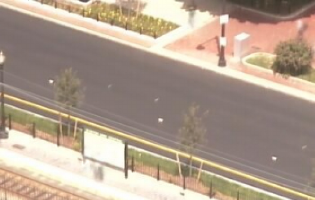}&	
			\includegraphics[width=2cm,height=1.6cm]{SBM/Our_boulevard.png} &	\includegraphics[width=2cm,height=1.6cm]{SBM/GT_boulevard-eps-converted-to.pdf} \\
			\hline
			\includegraphics[width=2cm,height=1.6cm]{SBM/ori/badminton_ori-eps-converted-to.pdf} &
			\includegraphics[width=2cm,height=1.6cm]{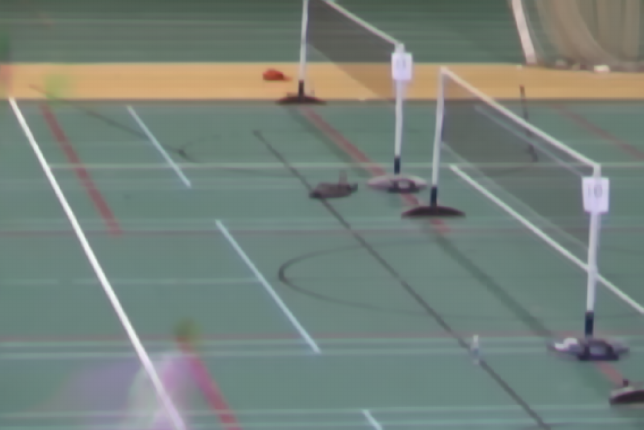} &
			\includegraphics[width=2cm,height=1.6cm]{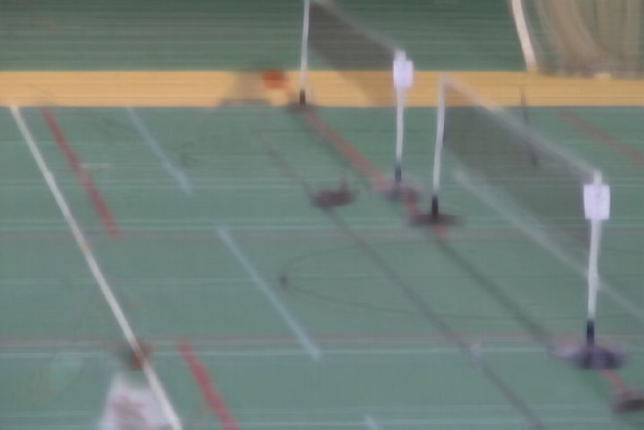} &
			\includegraphics[width=2cm,height=1.6cm]{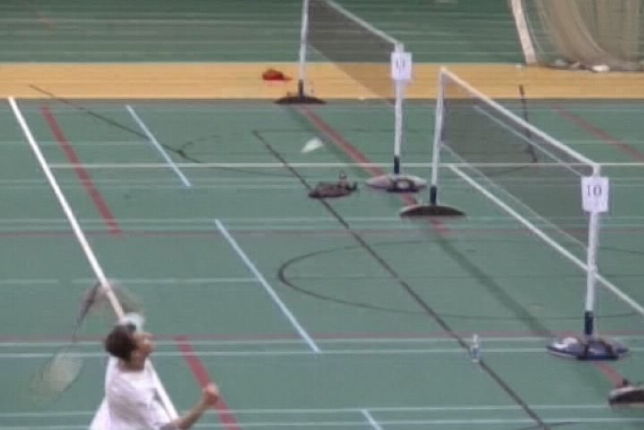} &
			\includegraphics[width=2cm,height=1.6cm]{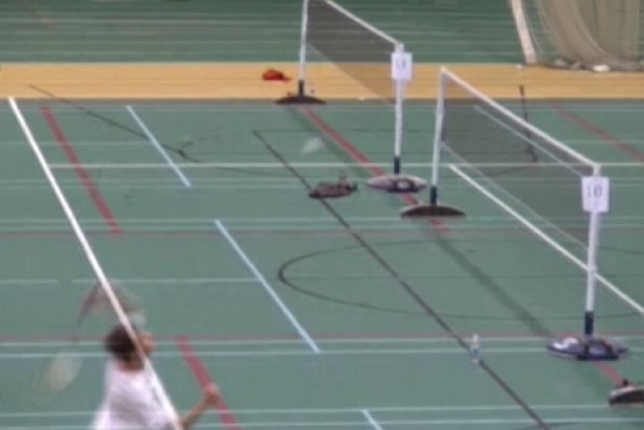} &
			\includegraphics[width=2cm,height=1.6cm]{SBM/Our_badminton.png}&	\includegraphics[width=2cm,height=1.6cm]{SBM/GT_badminton-eps-converted-to.pdf} \\
			\hline
			\includegraphics[width=2cm,height=1.6cm]{SBM/ori/busStation_ori-eps-converted-to.pdf} &
			\includegraphics[width=2cm,height=1.6cm]{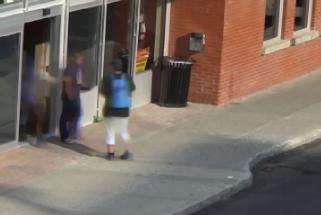} &
			\includegraphics[width=2cm,height=1.6cm]{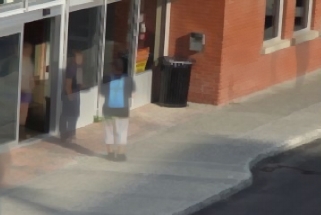} &
			\includegraphics[width=2cm,height=1.6cm]{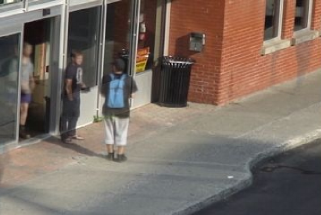} &
			\includegraphics[width=2cm,height=1.6cm]{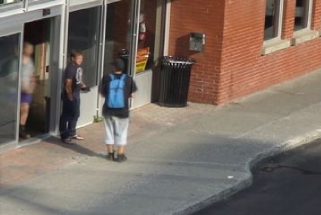}& 	
			\includegraphics[width=2cm,height=1.6cm]{SBM/Our_busStation.png}&
			\includegraphics[width=2cm,height=1.6cm]{SBM/GT_busStation-eps-converted-to.pdf} \\
			\multicolumn{1}{c}{{(a)}} & \multicolumn{1}{c}{{(b)}} & \multicolumn{1}{c}{{(c)}} & \multicolumn{1}{c}{{(d)}} & \multicolumn{1}{c}{{(e)}} & \multicolumn{1}{c}{{(f)}} & \multicolumn{1}{c}{{(g)}}  \\
		\end{tabular}
		\vspace{-0.2cm}
		\caption{\label{fig:SBMTV} Representative frames of six videos from the SBMnet dataset. (a) Original frames; (b)–(f) respectively show the background images generated by LRTV, TLR-SSTV, TCTV, SSCTV-RPCA, and our proposed method; (g) Ground truth.}
		\vspace{-0.3cm}
	\end{figure*}

	\begin{table*}[htbp]
		\footnotesize
		\caption{Quantitative evaluation on the SBMnet dataset: comparison of AGE, MSSSIM, PSNR,  and CQM. The bold numbers represent the best results for each metric on each video.}
		\centering
		\renewcommand\arraystretch{0.4}{
			\setlength{\tabcolsep}{3.0mm}{
				\begin{tabular}{lllllll}
					\toprule
					Videos &  & LRTV \cite{he2015total} & \makecell[c]{TLR- \\ SSTV} \cite{chen2018tensor} & TCTV \cite{WangHailin2023} & \makecell[c]{SSCTV- \\ RPCA} \cite{GAO2025128885} & \textbf{Ours} \\
					\midrule				
					\multirow{5}{*}{Basic (511)} & AGE$\downarrow$ & 6.5417 & 7.5416 & 5.8310 & 5.4763 & \textbf{3.7998}\\
					& MSSSIM$\uparrow$ & 0.9337 & 0.8999 & 0.9502 & 0.9515 & \textbf{0.9830}\\
					& PSNR$\uparrow$ & 25.8402 & 24.3712 & 27.3777 & 27.3557 & \textbf{30.7929}\\ 
					& CQM$\uparrow$ & 27.8452 & 26.2097 & 29.5372 & 29.4747 & \textbf{32.8066}\\ 			
					\midrule 
					\multirow{5}{*}{Basic (Blurred)} & AGE$\downarrow$ & 2.4874 & 5.6668 & 5.8126 & 5.4466 & \textbf{2.2926}\\
					& MSSSIM$\uparrow$ & 0.9933 & 0.9102 & 0.8976 & 0.9038 & \textbf{0.9939}\\
					& PSNR$\uparrow$ & 34.8302 & 24.4024 & 23.1763 & 23.7043 & \textbf{37.6021}\\ 				
					& CQM$\uparrow$ & 35.3560 & 25.4210 & 24.2387 & 24.7246 & \textbf{37.9613}\\ 
					\midrule
					\multirow{5}{*}{\makecell[l]{Clutter \\ (boulevardJam)}} & AGE$\downarrow$ & 3.0316 & 3.6233 & 4.2707 & 4.2205 & \textbf{2.4745}\\
					& MSSSIM$\uparrow$ & 0.9573 & 0.9363 & 0.8985 & 0.8975 & \textbf{0.9703}\\
					& PSNR$\uparrow$ & 31.5614 & 30.7066 & 28.4721 & 28.6762 & \textbf{32.9818}\\ 
					& CQM$\uparrow$ & 32.4626 & 31.8586 & 29.7340 & 29.9024 & \textbf{34.0861}\\ 
					\midrule
					\multirow{5}{*}{\makecell[l]{Jitter \\ (boulevard)}} & AGE$\downarrow$ & 15.5296 & 10.3665 & 15.8778 & 15.9789 & \textbf{9.1482}\\
					& MSSSIM$\uparrow$ & 0.7422 & 0.8697 & 0.7429 & 0.7354 & \textbf{0.9138}\\
					& PSNR$\uparrow$ & 18.5514 & 22.0058 & 18.6744 & 18.5079 & \textbf{22.8243}\\ 
					& CQM$\uparrow$ & 20.1310 & 23.4772 & 20.2264 & 20.0427 & \textbf{24.3108}\\ 
					\midrule
					\multirow{5}{*}{\makecell[l]{Jitter \\ (Badminton)}} & AGE$\downarrow$ & 10.1949 & 5.6367 & 12.1285 & 11.8425 & \textbf{3.5479}\\
					& MSSSIM$\uparrow$ & 0.6416 & 0.8456 & 0.5730 & 0.5804 & \textbf{0.9586}\\ 
					& PSNR$\uparrow$ & 22.3186 & 26.8973 & 20.8748 & 21.0902 & \textbf{31.8946}\\ 
					& CQM$\uparrow$ & 23.1666 & 27.6751 & 21.7829 & 21.9652 & \textbf{32.7904}\\ 
					\midrule
					\multirow{5}{*}{\makecell[l]{IntermittentMotion\\ (BusStation)}} & AGE$\downarrow$ & 9.6390 & 8.7636 & 9.7167 & 10.4638 & \textbf{4.7462}\\
					& MSSSIM$\uparrow$ & 0.8154 & 0.8489 & 0.8189 & 0.7987 & \textbf{0.9808}\\ 
					& PSNR$\uparrow$ & 20.3146 & 22.4911 & 20.7041 & 19.5868 & \textbf{32.2511}\\ 
					& CQM$\uparrow$ & 21.3159 & 23.5519 & 21.7803 & 20.6251 & \textbf{33.0502}\\ 
					\bottomrule
					\label{sbmtv}
				\end{tabular}
			}
		}
	\end{table*}

	\section{Conclusion}\label{sec:6}
	This paper introduced a novel theoretical framework, termed Tensor Orthogonal Subspace Split (TOSS), for splitting high-dimensional tensor data into two orthogonal components along a prescribed mode under an underlying subspace structure. By extending the classical idea of orthogonal projection from vector spaces to tensors in a consistent fiber-wise manner, TOSS establishes a principled mechanism for separating dominant subspace-aligned information from orthogonal residual variations. Several fundamental properties of the proposed framework were derived, and an important special case, rank-one TOSS, was further investigated. Building upon this theory, rank-one TOSS was applied to hyperspectral image restoration and color video background modeling, for which corresponding optimization models and efficient ADMM-based algorithms were developed. Experimental results demonstrated the effectiveness of the proposed methods and validated the practical value of the rank-one TOSS mechanism in these representative applications.
	
	The present study also suggests several promising directions for future research. First, although rank-one TOSS has been validated on two representative tasks, the broader potential of the general TOSS framework remains largely unexplored in concrete application scenarios. Extending the current theory to richer subspace configurations and verifying its effectiveness across a wider range of multidimensional data analysis problems would be of substantial interest. Second, the current models rely on a prescribed dominant direction, whereas a more adaptive mechanism that updates the principal direction within the optimization process may enable a more accurate and data-driven estimation of $\mathcal{X}_{+}$ and $\mathcal{X}_{-}$. Third, TOSS opens up a natural interface with deep learning: by designing dedicated yet cooperative neural architectures for the subspace-consistent component and the orthogonal residual component, it may be possible to combine the interpretability of the proposed theory with the expressive power of modern data-driven models. These directions point to a broader research agenda in which TOSS serves not only as a theoretical tool for structured tensor splitting, but also as a foundation for future developments in optimization-based modeling and learning-based multidimensional data analysis.

\bibliographystyle{IEEEtran}
\bibliography{references}

@article{kolda2009tensor,
	title={Tensor decompositions and applications},
	author={Kolda, Tamara G and Bader, Brett W},
	journal={SIAM Review},
	volume={51},
	number={3},
	pages={455--500},
	year={2009},
	publisher={SIAM}
}

@article{chen2018tensor,
	title={Tensor nuclear norm-based low-rank approximation with total variation regularization},
	author={Chen, Yongyong and Wang, Shuqin and Zhou, Yicong},
	journal={IEEE Journal of Selected Topics in Signal Processing},
	volume={12},
	number={6},
	pages={1364--1377},
	year={2018},
	publisher={IEEE}
}

@article{guo2022logarithmic,
	title={Logarithmic Schatten-$ p $ p norm minimization for tensorial multi-view subspace clustering},
	author={Guo, Jipeng and Sun, Yanfeng and Gao, Junbin and Hu, Yongli and Yin, Baocai},
	journal={IEEE Transactions on Pattern Analysis and Machine Intelligence},
	volume={45},
	number={3},
	pages={3396--3410},
	year={2022},
	publisher={IEEE}
}

@article{cichocki2015tensor,
	title={Tensor decompositions for signal processing applications: From two-way to multiway component analysis},
	author={Cichocki, Andrzej and Mandic, Danilo and De Lathauwer, Lieven and Zhou, Guoxu and Zhao, Qibin and Caiafa, Cesar and Phan, Huy Anh},
	journal={IEEE Signal Processing Magazine},
	volume={32},
	number={2},
	pages={145--163},
	year={2015},
	publisher={IEEE}
}

@article{de2000multilinear,
	title={A multilinear singular value decomposition},
	author={De Lathauwer, Lieven and De Moor, Bart and Vandewalle, Joos},
	journal={SIAM Journal on Matrix Analysis and Applications},
	volume={21},
	number={4},
	pages={1253--1278},
	year={2000},
	publisher={SIAM}
}

@article{harshman1970foundations,
	title={Foundations of the PARAFAC procedure: Models and conditions for an “explanatory” multi-modal factor analysis},
	author={Harshman, Richard A and others},
	journal={UCLA Working Papers in Phonetics},
	volume={16},
	number={1},
	pages={84},
	year={1970},
	publisher={Los Angeles, CA}
}

@article{oseledets2011tensor,
	title={Tensor-train decomposition},
	author={Oseledets, Ivan V},
	journal={SIAM Journal on Scientific Computing},
	volume={33},
	number={5},
	pages={2295--2317},
	year={2011},
	publisher={SIAM}
}

@article{zhao2016tensor,
	title={Tensor ring decomposition},
	author={Zhao, Qibin and Zhou, Guoxu and Xie, Shengli and Zhang, Liqing and Cichocki, Andrzej},
	journal={arXiv preprint arXiv:1606.05535},
	year={2016}
}

@inproceedings{zheng2021fully,
	title={Fully-connected tensor network decomposition and its application to higher-order tensor completion},
	author={Zheng, Yu-Bang and Huang, Ting-Zhu and Zhao, Xi-Le and Zhao, Qibin and Jiang, Tai-Xiang},
	booktitle={Proceedings of the AAAI Conference on Artificial Intelligence},
	volume={35},
	number={12},
	pages={11071--11078},
	year={2021}
}

@article{kilmer2013third,
	title={Third-order tensors as operators on matrices: A theoretical and computational framework with applications in imaging},
	author={Kilmer, Misha E and Braman, Karen and Hao, Ning and Hoover, Randy C},
	journal={SIAM Journal on Matrix Analysis and Applications},
	volume={34},
	number={1},
	pages={148--172},
	year={2013},
	publisher={SIAM}
}

@article{lu2019tensor,
	title={Tensor robust principal component analysis with a new tensor nuclear norm},
	author={Lu, Canyi and Feng, Jiashi and Chen, Yudong and Liu, Wei and Lin, Zhouchen and Yan, Shuicheng},
	journal={IEEE Transactions on Pattern Analysis and Machine Intelligence},
	volume={42},
	number={4},
	pages={925--938},
	year={2019},
	publisher={IEEE}
}

@inproceedings{lu2016tensor,
	title={Tensor robust principal component analysis: Exact recovery of corrupted low-rank tensors via convex optimization},
	author={Lu, Canyi and Feng, Jiashi and Chen, Yudong and Liu, Wei and Lin, Zhouchen and Yan, Shuicheng},
	booktitle={Proceedings of the IEEE Conference on Computer Vision and Pattern Recognition},
	pages={5249--5257},
	year={2016}
}

@article{bioucas2012hyperspectral,
	title={Hyperspectral unmixing overview: Geometrical, statistical, and sparse regression-based approaches},
	author={Bioucas-Dias, Jos{\'e} M and Plaza, Antonio and Dobigeon, Nicolas and Parente, Mario and Du, Qian and Gader, Paul and Chanussot, Jocelyn},
	journal={IEEE Journal of Selected Topics in Applied Earth Observations and Remote Sensing},
	volume={5},
	number={2},
	pages={354--379},
	year={2012},
	publisher={IEEE}
}

@article{yang2025subspace,
	title={Subspace-based coupled tensor decomposition for hyperspectral blind fusion},
	author={Yang, Kunjing and Bai, Minru and Dian, Renwei and Lu, Ting},
	journal={Inverse Problems and Imaging},
	volume={19},
	number={3},
	pages={560--591},
	year={2025},
	publisher={Inverse Problems and Imaging}
}

@book{Horn2012MatrixAnalysis,
	title={Matrix Analysis},
	author={Horn, Roger A. and Johnson, Charles R.},
	publisher={Cambridge University Press},
	year={2012},
	edition={2nd},
	pages={461--469}
}

@article{chang2020orthogonal,
	title={Orthogonal subspace projection-based go-decomposition approach to finding low-rank and sparsity matrices for hyperspectral anomaly detection},
	author={Chang, Chein-I and Cao, Hongju and Chen, Shuhan and Shang, Xiaodi and Yu, Chunyan and Song, Meiping},
	journal={IEEE Transactions on Geoscience and Remote Sensing},
	volume={59},
	number={3},
	pages={2403--2429},
	year={2020},
	publisher={IEEE}
}

@article{zhou2019tensor,
	title={Tensor low-rank representation for data recovery and clustering},
	author={Zhou, Pan and Lu, Canyi and Feng, Jiashi and Lin, Zhouchen and Yan, Shuicheng},
	journal={IEEE Transactions on Pattern Analysis and Machine Intelligence},
	volume={43},
	number={5},
	pages={1718--1732},
	year={2019},
	publisher={IEEE}
}

@article{sedighin2024tensor,
	title={Tensor methods in biomedical image analysis},
	author={Sedighin, Farnaz},
	journal={Journal of Medical Signals \& Sensors},
	volume={14},
	number={6},
	pages={16},
	year={2024},
	publisher={Medknow}
}

@article{khaleghi2013multisensor,
	title={Multisensor data fusion: A review of the state-of-the-art},
	author={Khaleghi, Bahador and Khamis, Alaa and Karray, Fakhreddine O and Razavi, Saiedeh N},
	journal={Information fusion},
	volume={14},
	number={1},
	pages={28--44},
	year={2013},
	publisher={Elsevier}
}

@article{neal2011distributed,
	title={Distributed optimization and statistical learning via the alternating direction method of multipliers},
	author={Neal, Parikh and Eric, Chu and Borja, Peleato and Jonathan, Eckstein},
	journal={Foundations and Trends{\textregistered} in Machine learning},
	volume={3},
	number={1},
	pages={1--122},
	year={2011},
	publisher={Emerald Publishing Limited}
}

@article{han2023multi,
	title={Multi-dimensional data recovery via feature-based fully-connected tensor network decomposition},
	author={Han, Zhi-Long and Huang, Ting-Zhu and Zhao, Xi-Le and Zhang, Hao and Liu, Yun-Yang},
	journal={IEEE Transactions on Big Data},
	volume={10},
	number={4},
	pages={386--399},
	year={2023},
	publisher={IEEE}
}

@article{auddy2025tensors,
	title={Tensors in high-dimensional data analysis: Methodological opportunities and theoretical challenges},
	author={Auddy, Arnab and Xia, Dong and Yuan, Ming},
	journal={Annual Review of Statistics and Its Application},
	volume={12},
	number={1},
	pages={527--551},
	year={2025},
	publisher={Annual Reviews}
}

@article{TOKCAN2026110191,
	title = {Tensor decompositions for signal processing: Theory, advances, and applications},
	journal = {Signal Processing},
	volume = {238},
	pages = {110191},
	year = {2026},
	author = {Neriman Tokcan and Shakir Showkat Sofi and Van Tien Pham and Clémence Prévost and Sofiane Kharbech and Baptiste Magnier and Thanh Phuong Nguyen and Yassine Zniyed and Lieven {De Lathauwer}},
}

@article{heng2023robust,
	title={Robust low-rank tensor decomposition with the {L}2 criterion},
	author={Heng, Qiang and Chi, Eric C and Liu, Yufeng},
	journal={Technometrics},
	volume={65},
	number={4},
	pages={537--552},
	year={2023},
	publisher={Taylor \& Francis}
}

@article{wu2022tensor,
	title={Tensor wheel decomposition and its tensor completion application},
	author={Wu, Zhong-Cheng and Huang, Ting-Zhu and Deng, Liang-Jian and Dou, Hong-Xia and Meng, Deyu},
	journal={Advances in Neural Information Processing Systems},
	volume={35},
	pages={27008--27020},
	year={2022}
}

@inproceedings{nie2021adaptive,
	title={Adaptive Tensor Networks Decomposition.},
	author={Nie, Chang and Wang, Huan and Tian, Le},
	booktitle={BMVC},
	pages={148},
	year={2021}
}

@article{nie2023adaptive,
	title={Adaptive tensor networks decomposition for high-order tensor recovery and compression},
	author={Nie, Chang and Wang, Huan and Zhao, Lu},
	journal={Information Sciences},
	volume={629},
	pages={667--684},
	year={2023},
	publisher={Elsevier}
}

@article{liu2024adaptively,
	title={Adaptively topological tensor network for multi-view subspace clustering},
	author={Liu, Yipeng and Chen, Jie and Lu, Yingcong and Ou, Weiting and Long, Zhen and Zhu, Ce},
	journal={IEEE Transactions on Knowledge and Data Engineering},
	volume={36},
	number={11},
	pages={5562--5575},
	year={2024},
	publisher={IEEE}
}

@book{kreyszig1991introductory,
	title={Introductory functional analysis with applications},
	author={Kreyszig, Erwin},
	year={1991},
	publisher={John Wiley \& Sons}
}

@misc{nimark2012projection,
	title={The Projection Theorem},
	author={Nimark, Kristoffer},
	year={2012}
}

@article{liu2023survey,
	title={A survey on hyperspectral image restoration: From the view of low-rank tensor approximation},
	author={Liu, Na and Li, Wei and Wang, Yinjian and Tao, Ran and Du, Qian and Chanussot, Jocelyn},
	journal={Science China Information Sciences},
	volume={66},
	number={4},
	pages={140302},
	year={2023},
	publisher={Springer}
}

@article{huang2015provable,
  title={Provable models for robust low-rank tensor completion},
  author={Huang, Bo and Mu, Cun and Goldfarb, Donald and Wright, John},
  journal={Pacific Journal of Optimization},
  volume={11},
  number={2},
  pages={339--364},
  year={2015}
}

@article{he2015total,
  title={Total-variation-regularized low-rank matrix factorization for hyperspectral image restoration},
  author={He, Wei and Zhang, Hongyan and Zhang, Liangpei and Shen, Huanfeng},
  journal={IEEE Transactions on Geoscience and Remote Sensing},
  volume={54},
  number={1},
  pages={178--188},
  year={2015},
  publisher={IEEE}
}

@article{WangHailin2023,
  author={Wang, Hailin and Peng, Jiangjun and Qin, Wenjin and Wang, Jianjun and Meng, Deyu},
  journal={IEEE Transactions on Pattern Analysis and Machine Intelligence}, 
  title={Guaranteed Tensor Recovery Fused Low-rankness and Smoothness}, 
  year={2023},
  volume={45},
  number={9},
  pages={10990-11007}}

@article{GAO2025128885,
title = {Hyperspectral image restoration via RPCA model based on spectral-spatial correlated total variation regularizer},
journal = {Neurocomputing},
volume = {619},
pages = {128885},
year = {2025},
issn = {0925-2312},
author = {Junheng Gao and Hailin Wang and Jiangjun Peng}}

@ARTICLE{5705575,
  author={Zhang, Lin and Zhang, Lei and Mou, Xuanqin and Zhang, David},
  journal={IEEE Transactions on Image Processing}, 
  title={FSIM: A Feature Similarity Index for Image Quality Assessment}, 
  year={2011},
  volume={20},
  number={8},
  pages={2378-2386}}

@ARTICLE{1284395,
  author={Zhou Wang and Bovik, A.C. and Sheikh, H.R. and Simoncelli, E.P.},
  journal={IEEE Transactions on Image Processing}, 
  title={Image quality assessment: from error visibility to structural similarity}, 
  year={2004},
  volume={13},
  number={4},
  pages={600-612}}

@book{wald2002data,
  author    = {Wald, Lucien},
  title     = {Data Fusion: Definitions and Architectures: Fusion of Images of Different Spatial Resolutions},
  publisher = {Presses des MINES},
  year      = {2002},
  address   = {Paris},
 }

@article{jodoin2017extensive,
	title={Extensive benchmark and survey of modeling methods for scene background initialization},
	author={Jodoin, Pierre-Marc and Maddalena, Lucia and Petrosino, Alfredo and Wang, Yi},
	journal={IEEE Transactions on Image Processing},
	volume={26},
	number={11},
	pages={5244--5256},
	year={2017},
	doi= {https://doi.org/10.1109/TIP.2017.2728181}
}

@article{BOUWMANS20173,
title = {Scene background initialization: A taxonomy},
journal = {Pattern Recognition Letters},
volume = {96},
pages = {3-11},
year = {2017},
note = {Scene Background Modeling and Initialization},
issn = {0167-8655},
author = {Thierry Bouwmans and Lucia Maddalena and Alfredo Petrosino}}

@article{Yalman2013,
  author = {Yalman, Yildiray and ERT{\"U}RK, {\.I}SMA{\.I}L},
  title = {A new color image quality measure based on {YUV} transformation and {PSNR} for human vision system},
  journal = {Turkish Journal of Electrical Engineering and Computer Sciences},
  volume = {21},
  number = {2},
  pages = {603--612},
  year = {2013},
   publisher = {The Scientific and Technological Research Council of Turkey}
}

@article{ramirez2017temporal,
	title={Temporal weighted learning model for background estimation with an automatic re-initialization stage and adaptive parameters update},
	author={Ramirez-Alonso, Graciela and Ramirez-Quintana, Juan A and Chacon-Murguia, Mario I},
	journal={Pattern Recognition Letters},
	volume={96},
	pages={34--44},
	year={2017}
}

@article{avola2017adaptive,
	title={Adaptive bootstrapping management by keypoint clustering for background initialization},
	author={Avola, Danilo and Bernardi, Marco and Cinque, Luigi and Foresti, Gian Luca and Massaroni, Cristiano},
	journal={Pattern Recognition Letters},
	volume={100},
	pages={110--116},
	year={2017}}

@article{2014Dynamic,
  title={Dynamic Mode Decomposition for Real-Time Background/Foreground Separation in Video},
  author={ Grosek, Jacob  and  Kutz, J. Nathan },
  journal={Computer Science},
  year={2014},
}

@article{HAN2022103560,
title = {Quaternion-based dynamic mode decomposition for background modeling in color videos},
journal = {Computer Vision and Image Understanding},
volume = {224},
pages = {103560},
year = {2022},
issn = {1077-3142},
author = {Juan Han and Kit Ian Kou and Jifei Miao},
}

@ARTICLE{9170824,
  author={Gao, Quanxue and Zhang, Pu and Xia, Wei and Xie, Deyan and Gao, Xinbo and Tao, Dacheng},
  journal={IEEE Transactions on Pattern Analysis and Machine Intelligence}, 
  title={Enhanced Tensor RPCA and its Application}, 
  year={2021},
  volume={43},
  number={6},
  pages={2133-2140}}

@article{YAN2024104520,
title = {Tensor robust principal component analysis via dual lp quasi-norm sparse constraints},
journal = {Digital Signal Processing},
volume = {150},
pages = {104520},
year = {2024},
issn = {1051-2004},
author = {Tinghe Yan and Qiang Guo}}

@misc{strang2022introduction,
	title={Introduction to Linear Algebra},
	author={Strang, Gilbert},
	year={2022},
	publisher={SIAM}
}

@article{kilmer2011factorization,
	title={Factorization strategies for third-order tensors},
	author={Kilmer, Misha E and Martin, Carla D},
	journal={Linear Algebra and its Applications},
	volume={435},
	number={3},
	pages={641--658},
	year={2011},
	publisher={Elsevier}
}
	
\end{document}